\documentclass[11pt]{article}

\usepackage{latexsym}
\usepackage{amsfonts}
\usepackage{amscd}
\usepackage{amsmath, amsthm, amssymb}
\usepackage{bbm}
\usepackage[all]{xy}
\usepackage{cite}

\author{Thomas Baird}
\title{GKM-sheaves and nonorientable surface group representations}

\newtheorem{thm}{Theorem}[section]
\newtheorem{cor}[thm]{Corollary}
\newtheorem{lem}[thm]{Lemma}
\newtheorem{prop}[thm]{Proposition}

\theoremstyle{definition}

\newtheorem{define}{Definition}
\newtheorem{rmk}{Remark}
\newtheorem{example}{Example}

\newcommand{\ignore}[1]{}

\newcommand{\lie}[1]{\mathfrak{#1}}
\newcommand{\clie}[1]{\mathfrak{#1}}
\newcommand{\cok}{\mathrm{cok}}
\newcommand{\ann}{\mathrm{ann}}

\newcommand{\rk}{\mathrm{rk}}
\newcommand{\im}{\mathrm{im}}
\newcommand{\id}{\mathbbmss{1}}
\newcommand{\Z}{\mathbb{Z}}
\newcommand{\R}{\mathbb{R}}
\newcommand{\C}{\mathbb{C}}

\newcommand{\G}{\mathcal{G}}
\newcommand{\ti}[1]{\tilde{#1}}

\newcommand{\F}{\mathcal{F}}
\newcommand{\M}{\mathcal{M}}

\newcommand{\Rp}[3]{\mathcal{R}^{#2}_{#3}(#1)}
\newcommand{\Rpe}[1]{\mathcal{R}^{#1}}

\newcommand{\E}{\mathcal{E}}
\newcommand{\V}{\mathcal{V}}
\newcommand{\Pt}{\mathbb{P} (\Lambda)}

\newcommand{\supp}{\mathrm{Supp}}
\newcommand{\nice}{nice }
\newcommand{\Top}{\mathrm{Top}}

\begin{document}


\maketitle

\begin{abstract}

Let $T$ be a compact torus and $X$ a nice compact T-space (say a manifold or variety). We introduce a functor assigning to $X$ a \emph{GKM-sheaf} $\F_X$ over a \emph{GKM-hypergraph} $\Gamma_X$. Under the condition that $X$ is \emph{equivariantly formal}, the ring of global sections of $\F_X$ are identified with the equivariant cohomology,  $H_T^*(X;\C) \cong H^0(\F_X)$. We show that GKM-sheaves provide a general framework able to incorporate numerous constructions in the GKM-theory literature.
	
In the second half of the paper we apply these ideas to study the equivariant topology of the representation variety $\mathcal{R}_K := Hom( \pi_1(\Sigma), K)$ under conjugation by $K$, where $\Sigma$ is a nonorientable surface and $K$ is a compact connected Lie group. We prove that $\mathcal{R}_{SU(3)}$ is equivariantly formal for all $\Sigma$ and compute its equivariant cohomology ring. We also produce conjectural betti number formulas for some other Lie groups.

\end{abstract}


\section{Introduction}

\subsection{GKM-sheaves}\label{tobeusedrain}

The first goal of this paper is to develop a general formalism for calculating the equivariant cohomology of torus actions. Let $T \cong U(1)^r$ be a compact torus of rank $r$ and $X$ a $T$-space. We call $X$ \emph{nice} if it admits the structure of a finite $T$-CW complex\footnote{ This class of $T$-spaces include compact, smooth $T$-manifolds and compact, real algebraic subsets of linear $T$ representations \cite{park1998semialgebraic}.}.

The equivariant cohomology\footnote{We use complex coefficients throughout.} $H_T^*(X) =H_T^*(X;\C)$ is an graded algebra over the equivariant cohomology of a point $H_T^*(pt) = H^*(BT)$. A $T$-space $X$ is called \emph{equivariantly formal} if $H_T^*(X)$ is a free $H^*(BT)$-module. In this case there is an isomorphism of graded $H^*(BT)$-modules, $$H_T^*(X) \cong H^*(X)\otimes H^*(BT).$$ Many interesting examples of $T$-spaces are known to be equivariantly formal, including compact Hamiltonian $T$-manifolds and spaces whose ordinary cohomology is zero in odd degrees.

Consider the equivariant topological filtration 
$$ X^T = X_0 \subseteq X_1 \subseteq... \subseteq X_r = X$$
where $X_j$ is the union of $T$-orbits of dimension less than or equal to $j$. The pair $(X_1, X_0)$ is called the \emph{one-skeleton} of $X$. If $X$ is a nice, equivariantly formal $T$-space,  the \emph{Chang-Skjelbred Lemma} \cite{cs}\footnote{The conclusion of the Chang-Skjelbred Lemma holds under more general hypotheses, see Allday-Franz-Puppe \cite{allday2011equivariant}}(see also \cite{franz2007ecs}) identifies $H_T^*(X)$ with the kernel of the coboundary map of the pair 
\begin{equation}\label{gnpibnprgpabpfnvlakjsdnf}
	\delta: H_T^*(X_0) \rightarrow H_T^{*+1}(X_1,X_0).
\end{equation}
An important advantage of the Chang-Skjelbred Lemma over other methods of calculating equivariant cohomology (such as Morse theory) is that the isomorphism $H_T^*(X) \cong \ker(\delta)$ is an isomorphism of rings, not simply of graded vector spaces. It is thus desirable to have methods to calculate $\ker(\delta)$. 

A class of T-spaces called \emph{GKM manifolds}, introduced by Goresky-Kottwicz-MacPherson \cite{gkm} (see also \cite{guillemin20011}), are particularly amenable to this calculation. A GKM-manifold is a compact Hamiltonian $T$-manifold $X$ for which $X_0$ is a finite set, and $X_1$ is two dimensional. This forces $X_1$ to be a union of 2-spheres upon which $T$ acts by rotation through characters $\alpha: T \rightarrow U(1)$.  These spheres intersect at the their fixed points or ``poles".  The topology of $(X_1,X_0)$ is neatly encoded in a GKM-graph or moment graph; this is the finite graph with vertices corresponding to points in $X_0$, edges corresponding to spheres in $X_1$, and edges labelled by characters of $T$. The equivariant cohomology $H_T^*(X)$ can be calculated entirely in terms of this GKM-graph using linear algebra.

In the current paper, we generalize this GKM procedure to all nice $T$-spaces $X$. In \S \ref{gkmhypergraphssec} we define a functor $ X \mapsto \Gamma_X$ from the category of \nice T-spaces to the category of \emph{GKM-hypergraphs}, so that $\Gamma_X$ encodes the combinatorics of the one-skeleton $(X_1,X_0)$. In \S \ref{gkmsheavessec} we define the notion of a \emph{GKM-sheaf} over a GKM-hypergraph, and associate to any \nice $T$-space $X$ a GKM-sheaf $\F_X$ over $\Gamma_X$ such that the algebra of global sections $H^0(\F_X)$ is naturally isomorphic to $\ker(\delta)$ from (\ref{gnpibnprgpabpfnvlakjsdnf}). If $X$ is equivariantly formal, then 
\begin{equation}\label{shibngotno}
H^0(\F_X) \cong H_T^*(X)
\end{equation}
by the Chang-Skjelbred Lemma.

Using GKM-sheaves opens access to sheaf theoretic operations that have proven useful in calculations. We introduce several such operations in \S \ref{opsongkm}, including push-fowards, external tensor products, induction and convolution.

GKM-sheaves are related to and were inspired by the $\Gamma$-sheaves of Braden-MacPherson \cite{braden2001mgi}, which are used in the study of equivariant intersection cohomology. In \S \ref{exmplesf} we show that every \emph{pure} $\Gamma$-sheaf determines a GKM-sheaf in such a way that their modules of global sections are isomorphic. GKM-sheaves are also a useful framework for understanding constructions due to Guillemin-Zara \cite{guillemin2003egf} and Guillemin-Holm \cite{gh}, related to the topology of Hamiltonian actions, a point we address in \S \ref{exmplesf}.

\subsection{Representation varieties}

Let $\pi$ be the fundamental group of a manifold $M$ and let $K$ be a compact, connected Lie group of rank $r$. The space of homomorphisms $Hom(\pi, K)$ is called a \emph{representation variety}. The group $K$ acts by conjugation on $Hom(\pi, K)$ producing a topological quotient stack which is naturally isomorphic to the moduli stack of flat $K$-bundles over $M$.  The second half of this paper studies the equivariant cohomology ring $H^*_K(Hom(\pi, K))$ when $M = \Sigma$ is a non-orientable 2-manifold (In fact, we work with a wider class of varieties associated to a punctured surface that includes to $Hom(\pi, K)$ as a special case).

In \cite{baird2008msf}, the author derived formulas for the cohomology rings $ H_{SU(2)}^*( Hom(\pi, SU(2)))$, describing Poincar\'e polynomials and cup product structure (unless otherwise indicated, cohomology is singular with complex coefficients). This calculation was greatly facilitated by the following fact:

\begin{thm}
For $K = SU(2)$, the $K$-space $Hom(\pi,K)$ is equivariantly formal. 
\end{thm}
Recall that a \nice $K$-space $X$ is called equivariantly formal if $H_{K}^*(X)$ is free as a module over $H^*(BK)$ or equivalently, if $H^*_K(X) \cong H^*(X) \otimes_{\C} H^*(BK)$ as graded vector spaces. The second objective of the current paper is to test whether this equivariant formality property generalizes to other Lie groups.  We get a positive result for $K=SU(3)$. 

\begin{thm}\label{firstone}
	Let $\Sigma$ be the connected sum of $g+1$ copies of $\R P^2$ and let $\pi = \pi_1(\Sigma)$.
The $SU(3)$-space $Hom(\pi, SU(3))$ is equivariantly formal, with equivariant Poincar\'e series 
\begin{equation}\label{gnonpabprb}
\frac{ (1+t^3+t^5+t^8)^g + (1+t^2+t^4)(t^3+2t^4+t^5)^g}{(1-t^2)(1-t^4)(1-t^6)}.
\end{equation}
\end{thm}
Moreover, we obtain an explicit description of the cup product structure on $H_{SU(3)}^*(Hom(\pi, SU(3)))$. The formula (\ref{gnonpabprb}) for the Poincar\'e series was conjectured by Ho and Liu \cite{ho2008msf} and later proven by the author \cite{baird2009antiperfection} using the Morse theory of the Yang-Mills functional, but this approach concealed the cup product structure. Because GKM-theory respects the cup product structure we are able to prove Theorem \ref{firstone} by combining GKM methods with the previously obtained formula (\ref{gnonpabprb}).

In an earlier draft of this paper we claimed to show that the conjecture fails for $K=SU(5)$ when $\Sigma$ is a Klein bottle. This claim relied on computer calculations that are now in doubt, so the claim has been withdrawn. We are currently only able to prove that equivariantly formality does not generally hold for the wider class of representation varieties $\Rp{c}{g}{K}$ defined in \S \ref{rpvarforpsfut} (see \S \ref{typea3} and \S \ref{aorecdulrocedj}).

We briefly outline the strategy. Recall that if $T \subset K$ is a maximal torus and $W = N(T)/T$ the Weyl group, then for any $T$-space $X$, there is a canonical isomorphism $$ H_K(X) =H_T^*(X)^W,$$ so to a great extent the study of connected compact groups actions reduces to torus actions. 

Let $R := Hom(\pi, K)$. Because $T$ is maximal abelian in $K$ it follows easily that $$R_0 := Hom(\pi, K)^T = Hom(\pi, T).$$ The codimension one tori in $T$ that occur as stabilizer groups of the torus action are precisely the root 
hypertori of $T$ in $K$. Thus if we let $\Delta_+$ denote the set of positive roots of $K$, we have 
\begin{equation}\label{bldyvanl}
R_1 = \bigcup_{ \alpha \in \Delta_+} Hom(\pi,K_{\alpha})
\end{equation}
where $K_{\alpha}$ is the centralizer of the kernel of $\alpha$. The groups $K_{\alpha}$ have semisimple rank one, so admit a covering $U(1)^{r-1}\times SU(2) \rightarrow K_{\alpha}$, where $r$ is the rank of $K$. The representation variety $Hom(\pi,K_{\alpha})$ can then be described using the known properties of $Hom(\pi,SU(2))$ from \cite{baird2008msf}. Since we have a good understanding of the one-skeleton $(R_1, R_0)$, we are able to construct the GKM-sheaf $\F_R$.  If $H_T^*(R)$ is equivariantly formal, then necessarily $H^0(\F_R)$ must be a free module, so computing $H^0(\F_R)$ provides a test of the equivariant formality of $H_T^*(R)$. In the case of $K = SU(3)$, this is combined with Morse theoretic information to compute $H_T^*(R)$ as a ring.

{\bf Acknowledgements:} I am grateful to Frank Gounelas, Tara Holm, Frances Kirwan, Damiano Testa and Geordie Williamson for helpful discussions. I would also like to thank the referee for a very thorough reading and wise advice for improving the presentation. This work was supported by an NSERC postdoctoral fellowship and an NSERC discovery grant.

\section{GKM theory}

\subsection{Equivariant cohomology}

Fix a compact torus $T$ of rank $r$ with complexified Lie algebra $\lie{t} = Lie(T) \otimes \C$ and dual $\lie{t}^* = Hom(\lie{t}, \C)$. Let $A:= \C[\lie{t}] \cong S(\clie{t}^*)$ denote the symmetric algebra of polynomial functions on $\lie{t}$, graded so that $\clie{t}^*$ has degree $2$. As a graded ring, $A$ is naturally isomorphic to $H^*(BT)$. We say that a $T$-space $X$ is \emph{\nice}if it admits the structure of a finite $T-CW$ complex.

Let $X$ be a nice $T$-space and let $Y$ be a nice $T$-subspace. Let $\lie{t}_x$ denote the complexified infinitesimal stabilizer of a point $x \in X$. The \emph{Borel localization theorem} says that as a module over $A$, the support of $H_T^*(X,Y)$  satisfies

\begin{equation}\label{Borelloc}
\supp( H_T^*(X,Y)) \subseteq \bigcup_{x\in X \setminus Y} \clie{t}_x.
\end{equation} 
Notice that by the finiteness condition on $X$, the right hand side of (\ref{Borelloc}) is equal to a finite union of linear spaces so it is an algebraic set. 

An important special case is $Y = X_0 = X^T$. Combining the localization theorem with the long exact sequence of the pair $i: X_0 \hookrightarrow X$, we conclude that the kernel and cokernel of the homomorphism

\begin{equation}\label{glnbongnpabnr}
i^*: H_T^*(X) \rightarrow H_T^*(X_0)
\end{equation}
are torsion $A$-modules. If $X$ is equivariantly formal, then $i^*$ restricts to an isomorphism between $H_T^*(X)$ and the kernel of (\ref{gnpibnprgpabpfnvlakjsdnf}) $$\delta: H_T^*(X_0)\rightarrow H_T^{*+1}(X_1,X_0).$$ 

Recall the filtration $X_0 \subseteq X_1 \subseteq ... \subseteq X_r = X$ described in \S \ref{tobeusedrain}. This filtration is preserved by $T$, so it determines a spectral sequence $(E^{p,q}_r, d_r)$ converging to $H_T^*(X)$  (see \cite{franz2007ecs}). The first page of the spectral sequence satisfies $$E_1^{p,q} \cong H_T^{p+q}(X_p, X_{p-1})$$ and $d_1^{p,q}: E_1^{p,q} \rightarrow E_1^{p+1,q}$ equals the coboundary operator of the triple $(X_{p+1},X_p,X_{p-1})$. In particular, we identify $\ker(\delta)$ with the initial column $E_2^{0,*}$.

\begin{prop}\label{aorcborecb}
Let $X$ be a \nice T-space. There is a natural short exact sequence of graded $A$-algebras, $$0 \rightarrow Tor_A(H_T^*(X)) \rightarrow H_T^*(X) \stackrel{\psi}{\rightarrow} \ker(\delta) \rightarrow \cok(\psi) \rightarrow 0,$$ where $\cok(\psi)$ is an $A$-module with support in codimension $\geq 2$. If $X$ is equivariantly formal, then $\psi$ is an isomorphism.
\end{prop}

\begin{proof}
First, note that $\ker(\delta) \subset H^*(X_0) = H^*(X_0)\otimes A$ is torsion free, so $Tor_A(H_T^*(X)) \subseteq \ker(\psi)$. Next, observe that because the spectral sequence $E_r^{p,q}$ converges to $H_T^*(X)$, $\ker(\psi)$ admits a finite filtration whose associated graded object is isomorphic to a subquotient of the $A$-module $\oplus_{p=1}^r H_T^*(X_p, X_{p-1})$, which is a torsion module by the localization theorem (\ref{Borelloc}). We deduce that $\ker(\psi)$ is torsion and thus $\ker(\psi) = Tor_A(H_T^*(X))$.

Similarly, an associated graded version of $\cok(\psi)$ is equal to $$ \bigoplus_{p=2}^{r} \im(d_p: E_p^{0,*} \rightarrow E_p^{p,(*-p+1)} )$$
which is a subquotient of $\bigoplus_{p=2}^{r} H_T^*(X_p, X_{p-1})$, so it has support in codimension two by the localization theorem. 

If $X$ is equivariantly formal, then the Chang-Skjelbred Lemma says that $\psi$ is an isomorphism.
\end{proof}

\subsection{GKM-hypergraphs}\label{gkmhypergraphssec}

Let $$\Lambda := Hom( T, U(1))$$ be the weight lattice of $T$, which we think of as embedded in $\lie{t}^*$ in the usual way.
Let $\mathbb{P}(\Lambda)$ denote the set of non-zero weights modulo scalar multiplication or \emph{projective weights}. The elements of $\mathbb{P}(\Lambda)$ are in one-to-one correspondence with the codimension one subtori of $T$ by the rule $$\alpha \in \Pt \leftrightarrow \ker(\ti{\alpha}) \subset T$$
where $\ti{\alpha} \in \Lambda$ is a primitive representative of $\alpha$. We will generally be sloppy and write $\alpha$ in place of $\ti{\alpha}$.

\begin{define}\label{aloeublaorcebuao}
A \textbf{GKM-hypergraph} $\Gamma = (\V_{\Gamma}, \sim)$  consists of 
\begin{itemize}
\item[(i)] a finite set $\V_{\Gamma}$ called the \emph{vertices} of $\Gamma$,

\item[(ii)] an equivalence relation $\sim_{\alpha}$ on $\V_{\Gamma}$ for each projective weight $\alpha \in \Pt$.
\end{itemize}
\end{define}

We denote by $Part_{\Gamma}(\alpha)$ the set of non-empty equivalence classes of $\sim_{\alpha}$. We frequently drop the subscript $\Gamma$ and write $\V$ and $Part(\alpha)$ when there is little risk of confusion. 

A \emph{morphism} of GKM-hypergraphs $\phi \in Hom_{GKM}( \Gamma, \Gamma')$ is a map of sets  $$\phi: \V_{\Gamma}\rightarrow \V_{\Gamma'}$$ such that for $ v, w  \in  \V_{\Gamma} $ and $\alpha \in \Pt$,  $v \sim_{\alpha} w$ implies $\phi(v)\sim_{\alpha} \phi(w)$  .

The main motivating example of GKM-hypergraphs is the following.

\begin{define}\label{aloegbuloegkbcbj}
To any nice $T$-space $X$, we associate the GKM-hypergraph $\Gamma_X$ by
\begin{itemize}
\item[(i)] $\V_X = \V_{\Gamma_X} := \pi_0(X^T)$, is the set of path components of the fixed point set $X^T$,
\item[(ii)] If $v_1, v_2 \in \V_X$ then $v_1 \sim_{\alpha} v_2$ if and only if corresponding path components lie in the same path component of the fixed point set $X^{\ker(\alpha)}$.
\end{itemize} 
\end{define}
The function $X \mapsto \Gamma_X$ extends to a functor from the category of nice $T$-spaces and equivariant maps to the category of GKM-hypergraphs.

Let $\wp(\V)$  denote the power set of $\V$. We consider $Part(\alpha)$ to be a subset of $\wp(\V)$ ( by convention, the empty set $\emptyset \not\in Part(\alpha)$). To a GKM-hypergraph $\Gamma$ define the set of \emph{hyperedges} $\E_{\Gamma}$ by

$$ \E = \E_{\Gamma} := \{ (S, \alpha) \in \wp(\V) \times \Pt~|~S \in Part(\alpha) \}.$$
Projection defines maps

\begin{itemize} 
	\item$\alpha: \E \rightarrow \Pt$~~~~\text{the \emph{axial function}, and} 

\item $I: \E \rightarrow \wp(\V)$ ~~~\text{the \emph{incidence map}.}
\end{itemize}
We say a hyper edge $e \in \E$ is \emph{incident} to a vertex $v \in \V$ if  $v \in I(e)$. Every hyperedge is incident to a least one vertex. A hyperedge $e$ is called \emph{degenerate} if it is incident to only one vertex. We say that a GKM-hypergraph is \emph{discrete} if all edges are degenerate.  We use the term \emph{GKM-graph} to mean a GKM-hypergraph for which no hyperedge is incident to more than two vertices.

\begin{rmk}
Observe that a GKM-hypergraph can be reconstructed from the data $\V$, $\E$, $I$ and $\alpha$. With this justification, we will sometimes define a GKM-hypergraph by specifying $(\V, \E, I, \alpha)$.
\end{rmk}

Given $\alpha \in \Pt$, denote $\E^{\alpha} := \{ e \in \E | \alpha(e) = e\}$. This forms a partition

\begin{equation}\label{partizon}
\E =  \coprod_{\alpha \in \Pt} \E^{\alpha}.
\end{equation}
and the elements in $\E^{\alpha}$ correspond to equivalence classes of $\sim_{\alpha}$.
\begin{prop}\label{aoelbuclaoceb}
Let $X$ be a nice $T$-space. There is a natural injection
$$\E^{\alpha}_X \hookrightarrow \pi_0(X^{\ker(\alpha)}).$$
whose image is those path components of $X^{\ker(\alpha)}$ containing $T$-fixed points.
\end{prop}
\begin{proof}
This is clear from the definitions.
\end{proof}

\begin{rmk}
When considering the GKM-hypergraph $\Gamma_X$ associated to a nice $T$-space $X$, we use lower case letters $v$ and $e$ to denote vertices and hyperedges, and upper case letters $V$ and $E$ to denote the corresponding path components of $X^T$ and $X^{\ker(\alpha(e))}$ respectively.
\end{rmk}

\subsubsection{The topology of a GKM-hypergraph}

Given a GKM-hypergraph $\Gamma$, we define $\Top(\Gamma)$ to be the a topological space with underlying set $\V_{\Gamma} \cup \E_{\Gamma}$ and with \emph{basic open sets} 
\begin{itemize}
\item $U_v := \{v \}$ for $v \in \V_{\Gamma} $ and 
\item $U_e := \{e \} \cup I(e)$ for $e \in \E_{\Gamma}$.
\end{itemize}
Somewhat counterintuitively, vertices are open points and hyperedges are closed points in this topology.  We now prove that $\Top$ extends to a functor from the category of GKM-hypergraphs to the category of topological spaces.

\begin{prop}
Any GKM-morphism, $\phi: \Gamma \rightarrow \Gamma'$ induces a map of sets $\varphi: \E_{\Gamma} \rightarrow \E_{\Gamma'}$ preserving the axial function and commuting with the incidence map.
\end{prop}

\begin{proof} 
Let $e = (I(e), \alpha(e)) \in \E_{\Gamma}$ be a hyperedge of $\Gamma$. By the definition of GKM-morphism, there exists a unique $e' \in \E_{\Gamma'}$ such that $\alpha(e) =\alpha(e')$ and $I(e) \subseteq I(e')$. Defining $\varphi(e) = e'$ completes the proof.
\end{proof}
Abusing notation, we denote by $$\phi: \V_{\Gamma} \cup \E_{\Gamma} \rightarrow \V_{\Gamma'} \cup \E_{\Gamma'}$$ the map of sets restricting to $\phi: \V_{\Gamma} \rightarrow \V_{\Gamma'}$ and $\varphi: \E_{\Gamma} \rightarrow \E_{\Gamma'}$.

\begin{prop}\label{localcover}
Let $\phi: \Gamma \rightarrow \Gamma'$ be a morphism of GKM-hypergraphs. The preimage of any basic open set $U_{x} \subseteq \Top(\Gamma')$ is a disconnected union of basic open sets:
\begin{equation}\label{earcbiaoe} \phi^{-1}(U_{x}) = \coprod_{y \in \phi^{-1}(x)} U_y.
\end{equation}
In particular, $\phi: \Top(\Gamma) \rightarrow \Top(\Gamma')$ is a continuous map and $\Top$ defines a functor from the category of GKM-hypergraphs to the category of topological spaces.
\end{prop}

\begin{proof}
Clearly (\ref{earcbiaoe}) holds for singleton sets $U_v = \{v\}$ centered on vertices. For a hyperedge $e' \in \E_{\Gamma'}$, we have $U_{e'} = e' \cup I(e')$, so we must show that if $\phi(v) \in I(e')$ then there exists $e \in \E_{\Gamma}$ such that $v \in I(e)$ and $\phi(e) = e'$. 

Suppose that $\phi(v) \in I(e')$. By the partition condition on GKM-hypergraphs, there is a unique $e \in \E$ such that $\alpha(e) = \alpha(e')$ and $v\in I(e)$. Thus $ \phi(v) \in \phi( I(e)) \subseteq I(\phi(e))$ and we deduce that $\phi(e) = e'$. Functoriality is clear.
\end{proof}

\subsection{GKM-sheaves}\label{gkmsheavessec}

We begin with an abstract definition.
\begin{define}\label{GKMsheaf}
Let $\Gamma$ be a GKM-hypergraph. A \textbf{GKM-sheaf} over $\Gamma$ is a sheaf $\F$ of finitely generated, $\Z$-graded $A$-modules over $\Top(\Gamma)$ satisfying the following three conditions.
\begin{enumerate}
\item $\F$ is locally free (that is, for every basic open set $U_x \in \Top(\Gamma)$, the stalk $\F(U_x)$ is a free $A$-module).	
	
\item For all $e \in \E_{\Gamma}$, the restriction map $res_e: \F(U_e) \rightarrow \F(I(e))$ becomes an isomorphism upon inverting $\alpha(e)$: \begin{equation}\F(U_e) \otimes_A A[\alpha(e)^{-1}] \cong \F(I(e)) \otimes_A A[\alpha(e)^{-1}].\end{equation}

\item $res_e: \F(U_e) \rightarrow \F(I(e))$ is an isomorphism for all but a finite number of $e \in \E_{\Gamma}$.
\end{enumerate}
We denote by $GKM(\Gamma)$ the full subcategory of the category of sheaves of graded $A$-modules on $\Top(\Gamma)$, whose objects are GKM-sheaves.
\end{define}
A nice $T$-space $X$ determines a sheaf of graded $A$-algebras on $\Top(\Gamma_X)$, denoted $\F_X$, with stalks $$\F_{X}(U_V)= \F_X(V) = H_T^*(V)$$ at vertices $V \subseteq X^T$ and $$\F_{X}(U_E) = \F_X(E \cup I(E)) = H_T^*(E) / Tor_A(H_T^*(E))$$ at edges $E \subseteq X^{\ker(\alpha(E))}$,  where $Tor_A(M) := \{ m \in M |~ a m = 0 \text{ for some } a \in A \setminus \{ 0\} \}$ is the torsion submodule of the $A$-module $M$. If $i: V \hookrightarrow E$ is a subset inclusion then the restriction map $ \F_{X}(U_E) \rightarrow \F_{X}(U_V)$ is identified with the cohomology morphism $i^*: H_T^*(E) \rightarrow H_T^*(V)$. This is well defined because $H_T^*(V)$ is torsion-free so $Tor_A(H_T^*(E)) \subseteq \ker(i^*)$.

\begin{prop}\label{freelocfree}
If $X$ is a \nice $T$-space then $\F_X$ is a GKM-sheaf. 
\end{prop}

\begin{proof}	
For any $E \in \E$, the restriction map $res_E: \F_{X}(U_E) \rightarrow \F(I(E))$ is identified with the map $H_T^*(E)/Tor_A(H_T^*(E)) \rightarrow H_T^*(E^T)$ induced by the inclusion $E^T \subset E$. Since the only isotropy algebra for $E \setminus E^T$ is $(\alpha_E)^{\perp} \subset \lie{t}$, the Borel localization theorem (\ref{Borelloc}) tells us that the kernel and cokernel of $res_E$ are $\alpha(E)$-torsion. The finiteness condition follows easily from compactness of $X$. Local freeness of $\F_X$ is an immediate consequence of Lemma \ref{locfreelem}.
\end{proof}

\begin{lem}\label{locfreelem}
If $X$ is a \nice $T$-space and $H\subset T$ is a codimension one subtorus, then $H_T^*(X^H)$ is the direct sum of a free and a torsion $A$-module. If $H^*_T(X)$ is torsion free, then $H^*_T(X^H)$ is free.
\end{lem}

\begin{proof}
Let $\alpha \in \Lambda \subset \lie{t}^*$ be the character for which $H$ is the kernel. Because $H$ acts trivially on $X^H$, $H_T^*(X^H) \cong H_{T/H}^*(X^H) \otimes_{\C[\alpha]} A$. Since $\C[\alpha]$ is a PID, the fundamental theorem of finitely generated modules over a PID implies that  $H_{T/H}^*(X^H)$ (hence also $H_T^*(X^H)$) is isomorphic to the sum of a free module and a torsion module, proving the first statement.

Furthermore, because $H_{T/H}^*(X^H)$ is a graded $\C[\alpha]$-module, its torsion submodule decomposes into a direct sum of modules of the form $$ \C[\alpha]/ \alpha^n \C[\alpha]$$ for some positive integer $n$ (up to degree shifts). Thus $H_T^*(X^H)$ is free if and only if it does not contain a summand isomorphic to $A / \alpha^n A$.

Assume that $H_T^*(X)$ is torsion free over $A$. Localizing at the prime ideal $(\alpha) \subset A$ we have $$H_T^*(X)_{(\alpha)} \cong H_T^*(X^{H})_{(\alpha)}$$ 
by the localization theorem. Localization is an exact functor so it preserves torsion-freeness and we infer that $H_T^*(X^{H})_{(\alpha)}$ is torsion free over $A_{(\alpha)}$. It follows that $H_T^*(X^{H})$ cannot contain a summand of the form $A / \alpha^n A$, so we conclude that $H_T^*(X^H)$ is free over $A$.
\end{proof}

Notice that $\F_X$ is in fact a sheaf of $A$-algebras, so the set of global sections $H^0(\F_X)$ is an $A$-algebra. Recall from \S \ref{tobeusedrain} the one-skeleton $(X_1,X_0)$ of a $T$-space $X$.

\begin{prop}\label{globalsex}
Let $X$ be a \nice $T$-space. The space of global sections $H^0(\F_X)$ fits into an exact sequence of graded $A$-modules
$$0\rightarrow H^0(\F_X) \stackrel{r}{\rightarrow} H_T^*(X_0) \stackrel{\delta}{\rightarrow} H_T^{*}(X_1,X_0)[+1].$$
for which $r$ is a homomomorphism of $A$-algebras.
\end{prop} 

\begin{proof}
Let $\Gamma_X = (\V,\E,I,\alpha)$. The map $r$ is identified with the sheaf restriction map $H^0(\F_X) \rightarrow \F_X(\V) \cong H^*_T( X_0)$, so $r$ is certainly a homomorphism of algebras. 

Since $\F_X$ is locally free and the restriction maps $res_E: \F_X(U_E) \rightarrow \F_X(I(E))$ are isomorphisms modulo torsion, they must be injective. Thus any element of $\F_X(\V)$ extends in at most one way to each hyperedge and we deduce that $r$ is injective.

It remains to show that $\im(r) = \ker(\delta)$. Decomposing cohomology into connected components we have an isomorphism $$ H_T^*(X_0) \cong \bigoplus_{V \in \V} H_T^*(V) $$
and $$ H_T^*(X_1,X_0) \cong \Big( \bigoplus_{E \in \E } H_T^*(E,E^T) \Big) \oplus  H_T^*(X_1') $$
where $X_1'$ is the union of components of $X_1$ that do not intersect $X_0$. Clearly the projection of $\delta$ onto $H_T^*(X_1')$ is zero, so $\ker(\delta) \cong \ker(d)$ where $d$ is the block decomposition 

\begin{equation}\label{combsheaf}
d: \bigoplus_{V \in \V} H_T^*(V) \rightarrow \bigoplus_{E \in \E} H_T^{*+1}(E,E^T).
\end{equation}
with matrix blocks $d_{E,V}$, such that $d_{E,V} = 0$ when $V\not\subseteq E$, and when $V \subseteq E$ the diagram

\begin{equation}\begin{CD}
	\xymatrix{  H_T^*(V) \ar[rr]^{d_{E,V}} \ar[dr]^{\iota} & &  H_T^{*+1}(E,E^T)  \\
	             &  H_T^*(E^T) \ar[ur]^{\delta_E} &}
\end{CD}\end{equation}
commutes, where $\delta_E$ is the boundary map of the pair $(E,E^T)$ and $\iota$ is inclusion as a summand. Thus $(a_V) \in \bigoplus_{V\in \V} H_T^*(V)$ lies in $ \ker(d)$ if and only if for all $E \in \E$, 
\begin{equation}\label{atcabin1}
0 = \sum_{V \in \V} d_{E,V}(a_V) = \sum_{V \subseteq E} d_{E,V}(a_V) = \delta(\sum_{V \subseteq E} \iota(a_V)).
\end{equation}
By the long exact sequence of the pair $(E, E^T)$, (\ref{atcabin1}) holds if and only if  $ \sum_{V \subseteq E^T} \iota(a_V) \in H_T^*(E^T)$ lies in the image of $H_T^*(E) \rightarrow H_T^*(E^T)$ and this is equal to the image of $res_E: \F_X( U_E) \rightarrow \F(I(E))$. Thus $\ker (d)$ corresponds to those sections in $\F(\V)$ that can extend to every hyperedge and we conclude that $\im(r) = \ker(d) = \ker(\delta)$. 

\end{proof}

\begin{thm}\label{ifqfor}
Let $X$ be a \nice T-space. There is a natural short exact sequence of graded $A$-algebras, $$0 \rightarrow Tor_A(H_T^*(X)) \rightarrow H_T^*(X) \stackrel{\psi}{\rightarrow} H^0(\F_X) \rightarrow \cok(\psi) \rightarrow 0,$$ where $\cok(\psi)$ is an $A$-module with support in codimension $\geq 2$. If $X$ is equivariantly formal, then $\psi$ is an isomorphism $$H_T^*(X) \stackrel{\psi}{\cong} H^0(\F_X).$$
\end{thm}

\begin{proof}
This follows immediately from Propositions \ref{aorcborecb} and \ref{globalsex}.
\end{proof}

Following examples are equivariantly formal.

\begin{example}
Let $S$ be a finite discrete set on which $T$ acts trivially. Then $\F_S$ is simply the constant sheaf $A_{\Top(\Gamma_S)}$. We abuse notation and denote this special case $A_S$.
\end{example}

\begin{example}
The GKM-sheaf of a GKM-manifold has stalks $$\F_{X}(U_x) = H_T^*(pt) = A$$ when $x$ is a vertex or degenerate edge, and $\F_{X}(U_e) \cong H^*_T(S^2) \cong A \oplus A[2]$ as an $A$-module when $e$ is a nondegenerate edge (we use the grading shift convention $M[d]^* \cong M^{*-d}$). The restriction map at a nondegenerate edge is the map from $\F_{X}(U_e) \cong A \oplus A[2] $ to $\F_X(I(e)) \cong A \oplus A$ defined by the matrix
$$
	\left( \begin{array}{cc}
1 & \alpha(e)   \\
1 & -\alpha(e)  \end{array} \right).$$
\end{example}

\begin{example}\label{aoelublaoecbulcrbix}
Consider a compact, connected Lie group $K$ with maximal torus $T$ acting on $K$ by conjugation, and let $\Delta \subset \Pt$ be the set of roots of $K$, modulo scalar multiplication. The GKM-hypergraph $\Gamma_K$ is the unique GKM-hypergraph with one vertex $v$. The GKM-sheaf $\F_K$ has stalk $$\F_{K}(U_v) = H_T^*(T) \cong \wedge \clie{t}^* \otimes_{\C} A$$ at $v$. For every edge $e$ with $\alpha(e) \not\in \Delta$, we have $\F_K(U_e) = \F_K(U_v) = \wedge \clie{t}^* \otimes A$ with restriction map the identity.  For every edge $e$ with $\alpha(e) \in \Delta$, let $H_{\alpha(e)} \subset \lie{t}^*$ be the hyperplane fixed by the Weyl reflection $S_{\alpha(e)}$. Then 
$\F_{K}(U_e)$ is the $A$-submodule of $\clie{t}^* \otimes_{\C} A$ generated by $(\wedge H_{\alpha(e)}) \otimes A$ and $\alpha(e) \otimes \alpha(e)$ with sheaf restriction map the inclusion.    
\end{example}

There is another useful description of the global sections of a GKM-sheaf. For each $\alpha \in \Pt$, denote be $i^*_{\alpha}$ the restriction map from $\F(\V \cup \E^{\alpha})$ to $\F(\V)$. Note that by the finiteness condition in Definition \ref{GKMsheaf}, $i^*_{\alpha}$ is an isomorphism for all but a finite set of $\alpha \in \Pt$.

\begin{prop}\label{noodelsouptime}
The restriction map $i^*: H^0(\F) \rightarrow \F(\V)$ is injective, with image 
\begin{equation}\label{intequal}
\im(i^*) = \bigcap_{\alpha \in \Pt} \im (i^*_{\alpha}).
\end{equation}
\end{prop}

\begin{proof}
For each hyperedge $e$ the restriction morphisms $ \F(U_e) \rightarrow \F(I(e))$ is injective, because $\F(U_e)$ is a free $A$-module and the kernel must be torsion free. 
This implies that sections in $\F(\V)$ can extend in at most one way to $\F(\V \cup \E) = H^0(\F)$, which is equivalent to $i^*$ being injective. By the proof of Proposition \ref{globalsex}, this means both sides of (\ref{intequal}) are equal the set of sections in $\F(\V)$ that extend to global sections.
\end{proof}

\subsection{Operations on GKM sheaves}\label{opsongkm}

\subsubsection{Pushforwards}

Given a continuous map between topological spaces $f: X \rightarrow Y$ and a sheaf $\F$ on $X$, that the \emph{pushforward sheaf} $f_*(\F)$ is defined by the rule $f_*(\F)(U) = \F(f^{-1}(U))$ for open sets $U \subseteq Y$. 

\begin{prop}\label{pushfisGKM}
Let $f: \Gamma \rightarrow \tilde{\Gamma}$ be a morphism of $GKM$-hypergraphs and $\F$ a $GKM$-sheaf over $\Gamma$. The push-forward sheaf, $f_*(\F)$, is a GKM-sheaf over $\tilde{\Gamma}$ satisfying
\begin{equation}\label{pushforeq}
H^0(\F) = H^0(f_*(\F)).
\end{equation}
\end{prop}

\begin{proof}

For any $y \in \Top(\ti{\Gamma})$ we have by Proposition \ref{localcover} that $$ f_*(\F)(U_y) = \bigoplus_{x \in f^{-1}(y)} \F( U_x)$$ which is a direct sum of free modules, hence free. By the same proposition, for a hyperedge $\ti{e} \in \E_{\ti{\Gamma}}$, the restriction map $ res_{\ti{e}}: f_*(\F)(U_{\ti{e}}) \rightarrow f_*(\F)(I(\ti{e}))$ is identified with the direct sum of maps $$ \bigoplus_{e \in f^{-1}(\ti{e})} \Big(res_e: \F(U_e) \rightarrow \F(I(e))\Big) $$ which is an isomorphism modulo $\alpha(e)$ and is an isomorphism for all but finitely many $\ti{e} \in \E_{\ti{\Gamma}}$. Equation (\ref{pushforeq}) holds simply by the definition of the pushforward sheaf.
\end{proof}

\begin{prop}
For any GKM-hypergraph morphism $\phi: \Gamma_1 \rightarrow \Gamma_2$ the pushforward $\phi_*$ of sheaves defines a functor $\phi_*: GKM(\Gamma_1) \mapsto GKM(\Gamma_2)$. 
\end{prop}
\begin{proof}
This follows from the fact that push-forward is a functor for sheaves of $A$-modules.
\end{proof}

\begin{prop}\label{gnosotokog}
Let $\Phi: X \rightarrow Y$ be a $T$-equivariant map between \nice $T$-spaces  and let $\phi: \Gamma_X \rightarrow \Gamma_Y$ be the induced morphism of GKM-hypergraphs. There is a natural morphism $ h :\F_Y \rightarrow \phi_*(\F_X)$ for which the following diagram commutes
\begin{equation}\begin{CD}
	\xymatrix{ H^*_T(Y)\ar[d]  \ar[rrr]^{\Phi^*} &&& H^*_T(X) \ar[d] \\
	H^0( \F_Y)  \ar[rrr]^{H^0(h)} &&& H^0(\phi_*(\F_X))= H^0(\F_X) }
\end{CD}\end{equation}
where the vertical arrows come from Theorem \ref{ifqfor}.
\end{prop}

\begin{proof}

A $T$-equivariant map $\Phi:X \rightarrow Y$ restricts to a map of the pairs $(X_1, X_0) \rightarrow (Y_1,Y_0)$. By Proposition \ref{globalsex}, there is an induced map $l: H^0(\F_Y) \rightarrow H^0(\F_X)$ which fits into the commutative diagram in place of $H^0(h)$. It remains to find a GKM-sheaf morphism $h: \F_Y \rightarrow \phi_*(\F_X)$ such that $H^0(h) = l$. 

If $\ti{V}$ is a connected component of $Y_0 = Y^T$ representing a vertex of $\Gamma_Y$, then at the level of stalks we have
$$ \F_Y(U_{\ti{V}}) = H^*_T(\ti{V})  \rightarrow \bigoplus_{V \in \phi^{-1}(\ti{V})}H_T^*(V) = \phi_*(\F_X)(U_{\ti{V}}) $$
is the direct sum of maps $H_T^*(\ti{V}) \rightarrow H_T^*(V)$ induced by the restriction of $\Phi$. The morphism is defined at the stalk of hyperedges similarly, but using components of $Y^{\ker(\alpha)}$ and $X^{\ker(\alpha)}$.  
This sheaf morphism is natural with respect to the block decomposition from the proof of Proposition \ref{globalsex}, so $H^0(h) = l$.
\end{proof}

\begin{rmk}
It is worth pointing out that the pull-back of a GKM-sheaf under GKM-morphism is \emph{not} necessarily a GKM-sheaf. Consider for example, the case $T = S^1$, $\Gamma$ is the GKM-graph with two vertices $v_1$ and $v_2$ and one edge $e$, $\Gamma'$ is the GKM-graph with one vertex, and $\phi: \Gamma \rightarrow \Gamma'$ is the unique morphism.  Then the constant sheaf $A_{\Gamma'}$ is GKM over $\Gamma'$, but the pull-back $\F := \phi^*(A_{\Gamma'})$ is not because $ \F(U_e) = \F(\Gamma) = A$ and $\F(I(e)) = \F(v_1) \oplus \F(v_2) = A \oplus A$ have different ranks.
\end{rmk}

\subsubsection{Group actions}

Let $G$ be a finite group acting by automorphisms on a GKM-hypergraph $\Gamma$. Define the quotient GKM-hypergraph $\Gamma/G$ with vertex set $\V_{\Gamma/G} = \V_{\Gamma}/G$ and equivalence relations $[v] \sim_{\alpha} [w]$ if and only if $gv \sim_{\alpha} w$ for some $g \in G$. The quotient map $\V_{\Gamma} \mapsto \V_{\Gamma}/G$ determines a GKM-morphism $$\pi: \Gamma \mapsto \Gamma/G.$$

Denote by $GKM_G(\Gamma)$ the category of $G$-equivariant GKM-sheaves. For any $\F \in GKM_G(\Gamma)$ the pushforward $\pi_*(\F)$ is a sheaf of modules over the group ring $AG$, so we may decompose into $\C G$-isotypical components $$\pi_*(\F) \cong \bigoplus_{\chi \in \hat{G}} \pi_*(\F)^{\chi}, $$ indexed by the set $\hat{G}$ of irreducible, complex $G$-representations, where for an open set $U \subset \Top(\Gamma/G)$ $$\pi_*(\F)^{\chi}(U) := (\pi_*(\F)(U))^{\chi} \cong \chi \otimes_{\C} Hom_{\C G}( \chi, \pi_*(\F)(U)).$$ 

\begin{lem}\label{isotype}
For each $\chi \in \hat{G}$, the isotypical component $\pi_*(\F)^{\chi}$ is a GKM-sheaf over $\Gamma/G$ and there is a natural isomorphism
$$ H^0(\F)^{\chi} \cong H^0( \pi_*(\F)^{\chi}).$$ 
\end{lem}
\begin{proof}
Everything follows easily from the fact that summands of free $A$-modules are free $A$-modules and that $G$-equivariant maps respect isotypical components. 
\end{proof}

An important special case of Lemma \ref{isotype} is the $G$-invariant subsheaf which we denote $\pi_*(\F)^G$.

\begin{lem}\label{arcrcihcwmbjt}
Let $G$ be a finite group and let $X$ be a \nice $G \times T$-space so that $X/G$ is a \nice $T$-space. Then $\Gamma_X$ inherits a $G$-action and there is a natural isomorphism of GKM-sheaves
$$ \F_{X/G} \cong \pi_*(\F_X)^G $$
\end{lem}

\begin{proof}
This follows pretty directly from the  natural isomorphisms $H_T^*((X/G)^T) \cong H_T^*((X^T)/G) \cong H_T^*(X^T)^G$ and similarly for codimension one subtori $H \subset T$. 
\end{proof}

\subsubsection{Tensor products}

\begin{define}
For $i=1,2$, let $\Gamma_i = (\V_i, \sim)$ be GKM-hypergraphs. Define the \textbf{product GKM-hypergraph} $\Gamma_1 \times \Gamma_2$ with vertices $\V_1 \times \V_2$ and relations $(v_1,v_2) \sim_{\alpha} (w_1,w_2)$ if and only if $v_1\sim_{\alpha} w_1$ and $v_2 \sim_{\alpha} w_2$.
\end{define}

The projection maps $\pi_i: \Gamma_1 \times \Gamma_2 \rightarrow \Gamma_i$ are GKM-morphisms making $\Gamma_1 \times \Gamma_2$ a product object in the category of GKM-hypergraphs. 
 
Suppose now that $\Gamma_1$ and $\Gamma_2$ admit a $G$-actions.  Then we have a diagram of $GKM$-hypergraphs

$$ \xymatrix{  & \Gamma_1 \times \Gamma_2 \ar[dl]_{\pi_1} \ar[dr]^{\pi_2} \ar[d]^{\phi} & \\
                    \Gamma_1 & \Gamma_1 \times_G \Gamma_2 & \Gamma_2 } $$
where $\phi$ is the quotient map for the diagonal $G$-action on $\Gamma_1 \times \Gamma_2$.

\begin{prop}
Let $\F_i$ be a GKM-sheaf on $\Gamma_i$ for $i=1,2$. The \textbf{external tensor product} $$ \F_1 \boxtimes_G \F_2 := \phi_*(\pi_1^*(\F_1) \otimes_{A} \pi_2^*(\F_2))^G$$ is a GKM sheaf on $\Gamma_1 \times_G \Gamma_2$. This construction determines a bifunctor
$$GKM_G(\Gamma_1) \times GKM_G(\Gamma_2) \rightarrow GKM(\Gamma_1 \times_G \Gamma_2) $$
\end{prop}

\begin{proof}
This bifunctor factors into $GKM_G(\Gamma_1) \times GKM_G(\Gamma_2) \mapsto GKM_G(\Gamma_1 \times \Gamma_2)$
and the $G$-invariant pushforward $GKM_G(\Gamma_1 \times \Gamma_2) \mapsto GKM(\Gamma_1 \times_G \Gamma_2)$, so
by Lemma \ref{isotype} it suffices to consider the case that $G$ is trivial. 

The finiteness and locally free conditions clearly hold. It remains to prove that for all hyperedges $e$, $$ res_{e}: (\F_1 \boxtimes \F_2)(U_{e}) \rightarrow (\F_1 \boxtimes \F_2)(I(e))$$ is an isomorphism modulo $\alpha(e)$. It is an easy check that $e = (I(e_1)\times I(e_2), \alpha(e))$ for some hyperedges $e_1 \in \E_{\Gamma_1}$ and $e_2 \in \E_{\Gamma_2}$ with $\alpha(e)= \alpha(e_1) =\alpha(e_2)$.  By definition, $res_e$ is identified with
$$res_{e_1} \otimes res_{e_2}: \F_1(U_{e_1}) \otimes \F_2(U_{e_2}) \rightarrow \F_1(I(e_1)) \otimes \F_2(I(e_2)) $$
and $res_{e_i}$ is an isomorphism after inverting $\alpha(e) = \alpha(e_i)$ for $i=1,2$, so $res_e$ is also an isomorphism after inverting $\alpha(e)$.
\end{proof}

\begin{prop}
The external tensor product is associative. That is, if $\F_i \in GKM_{G_i \times G_{i+1}}(\Gamma_i) $ for $i=1,2,3$ then there is a natural isomorphism 
$$ (\F_1 \boxtimes_{G_2}\F_2) \boxtimes_{G_3} \F_3 \cong \F_1 \boxtimes_{G_2}(\F_2 \boxtimes_{G_3} \F_3). $$
in $GKM_{G_1\times G_4}(\Gamma_1 \times_{G_2} \Gamma_2 \times_{G_3} \Gamma_3)$
\end{prop}

\begin{proof}
Both sides are identified with the $G_2 \times G_3$-invariant pushforward of the sheaf 
$$ \pi_1^*\F_1 \otimes \pi_2^*\F_2 \otimes \pi_3^*\F_3 $$
where $\pi_i:  \Gamma_1 \times \Gamma_2 \times \Gamma_3 \rightarrow \Gamma_i$ are the projection maps.	
\end{proof}

\begin{prop}\label{clogbairdisagooddog}\label{touseorbit}
Suppose that $X$ and $Y$ are \nice $G\times T$-spaces for $G$ a finite group, and suppose that both $H_T^*(X^{\ker(\alpha)})$ and $H_T^*(Y^{\ker(\alpha)})$ are free $A$-modules for all $\alpha \in \Pt$ (this holds if $H_T^*(X)$ are $H_T^*(Y)$ are torsion free by Lemma \ref{locfreelem}). Then $$\F_X \boxtimes_G \F_Y \cong \F_{X\times_G Y}$$ for the diagonal $T$-action on $X \times_G Y$.
\end{prop}

\begin{proof}
By Lemma \ref{arcrcihcwmbjt}, it suffices to consider the case that $G$ is trivial.

It is evident that $(X\times Y)^T = X^T \times Y^T$ and that $X^{\ker(\alpha)}\times Y^{\ker(\alpha)} = (X \times Y)^{\ker(\alpha)}$ for all $\alpha \in \Pt$, so the corresponding equalities on connected components produces an isomorphism of GKM graphs $\Gamma_X \times \Gamma_Y \cong \Gamma_{X \times Y}$. 

The isomorphism of sheaves is defined through the Kunneth morphisms

\begin{equation*}\label{scheckt}
 H_T^*(X^T) \otimes_A H_T^*(Y^T) \rightarrow H_T^*(X^T \times Y^T)
\end{equation*}
\begin{equation*}\label{hotoffice}
H^*_T(X^{\ker(\alpha)}) \otimes_A H^*_T(Y^{\ker(\alpha)}) \rightarrow H^*_T((X \times Y)^{\ker(\alpha)})
\end{equation*}
which are isomorphisms because the modules are all free over $A$.
\end{proof}

\begin{prop}\label{tensisomrp}\label{wasptodie}\label{alcrldecrudi.}
Suppose that $\F_1$ and $\F_2$ are $G$-equivariant GKM-sheaves. Then the kernel and cokernel of the natural map 
\begin{equation}\label{tensisormpequ}
 \phi: H^0(\F_1) \displaystyle\otimes_{AG} H^0( \F_2) \rightarrow H^0(\F_1 \boxtimes_G \F_2)
\end{equation}
are torsion with support in codimension greater than one. If additionally $H^0(\F_k)$ is free over $A$ for $k=1,2$, then $\phi$ is an isomorphism.
\end{prop}

\begin{proof}

Because $H^0(\F_1) \displaystyle\otimes_{AG} H^0( \F_2) = (H^0(\F_1) \displaystyle\otimes_{A} H^0( \F_2))^G$ and $H^0(\F_1 \boxtimes_G \F_2) = H^0(\F_1 \boxtimes \F_2)^G$, it suffices to consider the case that $G$ is trivial.

We use Proposition \ref{noodelsouptime} to describe the global sections functor. Let $$i_{k,\alpha}^* : \F_k(\V_k \cup \E_k^{\alpha}) \rightarrow \F_k(\V_k)$$ denote the restriction map and let $\Delta \subset \Pt$ be the finite set of $\alpha$ for which $i_{k,\alpha}^*$ is not an isomorphism for either  $k=1$ or $2$. The map (\ref{tensisormpequ}) is equivalent to the functorial map

$$\psi: \Big(\bigcap_{\alpha \in \Delta} \im(i_{1,\alpha}^*)\Big)\otimes \Big(\bigcap_{\alpha \in \Delta} \im(i_{2,\alpha}^*)\Big) \rightarrow  \bigcap_{\alpha \in \Delta} \im(i_{1,\alpha}^* \otimes i_{2,\alpha}^*). $$

Our hypotheses imply that $i^*_{\alpha}$ is an injective map between free $A$-modules and that $i^*_{\alpha}$ becomes an isomorphism after localizing to the hyperplane $(\alpha)^{\perp} \subset \clie{t}$. If $\Delta$ has cardinality one, then $\psi$ is an isomorphism because everything is free. Since localization commutes with tensor products and finite intersections (\cite{eisenbud1995cav} \S 2.2), we deduce that if $x \in \clie{t}$ is annihilated by no more than one element of $\Delta$, then the localization of $\psi$ at $x$ is an isomorphism. Consequently, the support of both $ \ker(\psi)$ and $\cok(\psi)$ must lie in the union of codimension two planes $(\alpha_0)^{\perp} \cap (\alpha_1)^{\perp}$ where $\alpha_0$ and $\alpha_1$ vary over distinct elements of $\Delta$.

Now suppose that $H^0(\F_k)$ is a free $A$-module for $k=1,2$ and let $M$ and $N$ denote the source and target of $\psi$ respectively. Then $M$ is free, so $M \cong A^d$ for some non-negative intger $d$ (ignoring the grading). Because $\ker(\psi)$ is torsion submodule of $M$, it must be zero and we have a short exact sequence  

	\begin{equation}\label{bobhom}
	 0 \rightarrow M \stackrel{\psi}{\rightarrow} N \rightarrow \cok(\psi) \rightarrow 0.
	\end{equation}

	Now suppose for the sake of contradiction that $\cok(\psi)$ is nonzero. By the theory of associated primes (\cite{eisenbud1995cav}, \S 3.1) there exists some nonzero $y \in \cok( \psi)$ for which the annihilator $ \ann(y) = \{ a\in A| ay=0\}$ is a prime ideal. By the result on the support of $\cok(\psi)$, there must exist distinct elements $\alpha_0, \alpha_1 \in \Delta \cap \ann(y)$. We apply the functor $Tor_*( A/ \alpha_0 A , -)$ to (\ref{bobhom}). We know $N$ is torsion free because it injects into $(A^n)^{\otimes m}$, so we obtain an exact sequence

	$$ 0 \rightarrow Tor_1(A/\alpha_0 A, \cok(\psi)) \rightarrow (A/\alpha_0 A)^d$$
	where $y \in Tor_1(A/\alpha_0 A, \cok(\psi)) = \{ z \in \cok (\psi)| \alpha_0 z = 0\} $. This implies that $(A/\alpha_0 A)^d$ contains a nonzero element annihilated by $\alpha_1$ which is a contradiction. Thus $\cok(\psi) = 0$ and $\psi$ is an isomorphism.	
\end{proof}

\subsubsection{Induction}

Let $G$ be a finite group and consider the discrete GKM-hypergraph $\Gamma_G$ with vertex set $G$. Then $G\times G$ acts on $\Gamma_G$ by left and right multiplication, and the constant sheaf $A_{G}:= A_{\Top(\Gamma_G)}$ is a $G\times G$-equivariant GKM-sheaf. 

Given a homomorphism $H \rightarrow G$, we  make $A_{G}$ a $G \times H$-equivariant.  If a GKM-hypergraph $\Gamma$ has left $H$-action, define the induction functor

$$ Ind_H^G: GKM_H(\Gamma) \rightarrow GKM_G( G\times_H \Gamma) $$
by $Ind_H^G(\F) = A_G \boxtimes_H \F$.  This is a well-defined functor because it is obtained from the external tensor product. Note that $Ind_H^G(\F)$ depends on the homomorphism $\phi: H \rightarrow G$ even though it has been suppressed in the notation. For the identity morphism on $G$ we have natural isomorphism $G\times_G \Gamma = \Gamma = \Gamma \times_G G$ and

\begin{equation}\label{arclintg}
Ind_G^G(\F) = A_G \boxtimes_G \F \cong \F \cong \F \boxtimes_G A_G.
\end{equation}

\subsubsection{Convolution}\label{aocegbulagcobj}

Suppose now that $\Gamma$ is a GKM-hypergraph equipped with the action of an \emph{abelian} group $G$ which is free and transitive on the set of vertices.   

\begin{lem}\label{alrcedlacrud}
There is an isomorphism of GKM-hypergraphs $$\phi: \Gamma \cong \Gamma \times_G \Gamma,$$ which is canonically defined up to a choice of base vertex in $\Gamma$. Furthermore, $\phi$ is equivariant with respect to a residual $G$-action on $\Gamma \times_G \Gamma$.
\end{lem}

\begin{proof}
Let $\Gamma = (\V, \sim)$ and choose a base vertex $v_* \in \V$. Define a morphism $\phi: \Gamma \rightarrow \Gamma \times_G \Gamma$ as the composition $\pi \circ i$

$$ \xymatrix{ \Gamma \ar[r]^i \ar[dr]& \Gamma \times \Gamma \ar[d]^{\pi} \\
  &  \Gamma \times_G \Gamma } $$
where $i(v) = (v, v_*)$ and $\pi$ is the quotient map. It is easy to see that $\phi$ is a bijective on vertices, because every $G$ orbit in $\V \times \V$ passes once through $\V\times \{v_*\}$. To see that $\phi$ is a $GKM$-isomorphism, it is enough to count hyperedges $\E^{\alpha}$ and $\E^{\alpha}\times_G \E^{\alpha}$ for each $\alpha \in \Pt$ and show they have the same cardinality. 

Because $G$ acts transitively on $\V$, it also acts transitively on $\E^{\alpha}$. Because $G$ is abelian the stabilizer $G_{e}$ is the common stabilizer of all hyperedges $e \in \E^{\alpha}$ and thus also the common stabilizer of all elements in $\E^{\alpha} \times \E^{\alpha}$ under the anti-diagonal action. Consequently by the orbit-stabilizer theorem, we have $$| (\E^{\alpha} \times \E^{\alpha})/G | = | \E^{\alpha} \times \E^{\alpha} |/ |\E^{\alpha} | = |\E^{\alpha} | $$
proving that $\pi \circ i$ is an isomorphism.

Because the diagonal subgroup  $\Delta G \subseteq G\times G$ is normal, the $G \times G$-action on $\Gamma \times \Gamma$ descends to a  $(G \times G/\Delta G)$-action on $\Gamma \times_G \Gamma$. Under the isomorphism $$G \cong (G \times G)/ \Delta G,~~g \mapsto [(g,id_G)],$$ $\phi$ becomes $G$-equivariant. 

\end{proof}

We call the composition of functors

 $$ GKM_G(\Gamma) \times GKM_G(\Gamma) \mapsto GKM_G(\Gamma \times_G \Gamma) \cong GKM_G(\Gamma) $$
the \emph{convolution product}, denoted $*$.

\begin{rmk} The convolution product as defined is not symmetric in general.  This can be rectified by working instead with $\Gamma^{op}\times_G \Gamma$, where $\Gamma^{op} = \Gamma$ with $G$ acts through the inverse map $g \mapsto g^{-1}$. In our applications $G$ is a $2$-torsion group, so the inverse map is the identity and this issue is irrelevant.
\end{rmk}

\begin{lem}\label{overload}
For $i=1,2$ let $\Gamma_i$ be a GKM-hypergraph with an action by a abelian group $G_i$ which is free and transitive on vertices. The external tensor product determines a bifunctor $$ \boxtimes: GKM_{G_1}(\Gamma_1) \times GKM_{G_2}(\Gamma_2) \rightarrow GKM_{G_1\times G_2}(\Gamma_1 \times \Gamma_2),$$ for which there are natural isomorphisms
\begin{equation}\label{onthebus}
(\F_1 * \G_1) \boxtimes (\F_2 * \G_2) \cong (\F_1 \boxtimes \F_2) * (\G_1 \boxtimes \G_2)
\end{equation}
for $\F_i, \G_i \in GKM_{G_i}(\Gamma_i)$.  
\end{lem}

\begin{proof}

The product operation of GKM-hypergraphs is associative and commutative, so $ (\Gamma_1 \times \Gamma_1) \times ( \Gamma_2 \times \Gamma_2)$ is naturally isomorphic to $ (\Gamma_1 \times \Gamma_2) \times ( \Gamma_1 \times \Gamma_2)$.  Both sides of (\ref{onthebus}) are invariant pushforwards of the naturally isomorphic GKM-sheaves $ (\F_1 \boxtimes \G_1) \boxtimes (\F_2 \boxtimes \G_2) \cong (\F_1 \boxtimes \F_2) \boxtimes (\G_1 \boxtimes \G_2)$.
\end{proof}

\begin{prop}\label{supeagone}
Let $\Gamma$ be GKM-graph equipped with a transitive and vertex-free action by a finite abelian group $H$, and let $\phi: H \rightarrow G$ be a homomorphism of finite abelian groups. Then the induced $G$-action on $G \times_H \Gamma$ is transitive and vertex-free, and the induction functor $Ind_H^G: GKM(\Gamma) \mapsto GKM(G \times_H \Gamma)$ respects the convolution product,
$$Ind_H^G( \F * \G) \cong Ind_H^G(\F)*Ind_H^G(\G) $$
\end{prop}

\begin{proof}
Applying (\ref{arclintg}) and associativity we have

\begin{eqnarray*}
Ind^G_H(\F * \G ) & =&  A_G \boxtimes_H (\F \boxtimes_H \G) \\
&= & A_G \boxtimes_H (\F \boxtimes_G A_G) \boxtimes_H \G\\
 &=& (A_G \boxtimes_H \F) \boxtimes_G (A_G \boxtimes_H \G) =Ind_H^G(\F) * Ind_H^G(\G)
\end{eqnarray*}

\end{proof}

\subsection{Twisted actions}\label{twsithgadctiongs}

Let $W$ be a finite group. Given a torus $T$ a \emph{twist} $\tau$ is a homomorphism $$\tau: W \rightarrow Aut(T),~w \mapsto \tau_w.$$ In all examples we consider, $W$ will be a Weyl group acting in the standard way on $T$. A twist induces actions of $W$ on $\Pt$ and on $A = S(\lie{t}^*)$ that we also denote by $\tau$.

A \emph{$\tau$-twisted action on a GKM-hypergraph} $\Gamma = (\V, \E, I, \alpha)$ consists of $W$-actions on the sets $\V$ and $\E$ that are equivariant with respect to $I$ and $\alpha: \E \rightarrow \Pt$. This differs from an ordinary GKM-action, because $W$ is allowed to act nontrivially on $\Pt$. 

A \emph{$\tau$-twisted $W$-action on a GKM sheaf} $\F$ over $\Gamma$ is a lift of a $\tau$-twisted $W$-action $\rho$ on $\Gamma$ to an $W$-action $\ti{\rho}$ on $\F$ as a sheaf of $\Z$-graded abelian groups, satisfying the identity $\ti{\rho}_w( f s) = \tau_w(f) \ti{\rho}_w(s)$ for all $w \in W$, $f \in A = S(\lie{t}^*)$ and sections $s$ of $\F$. Given such an action, $W$-invariant global sections $H^0(\F)^{W}$ form a graded $A^W$-module.

Let $K$ be a compact Lie group with maximal torus $T$ such that the normalizer $N_K(T)$ intersects every path component of $K$, and let $X$ be a \nice $K$-space. Restricting to the action of $T$, we may associate a GKM-hypergraph $\Gamma_X$ and GKM-sheaf $\F_X$. Since the one-skeleton $(X_1,X_0)$ is preserved by $N_K(T)$, we gain a twisted action of $W = N_K(T)/T$ on $\Gamma_X$ lifting to $\F_X$. The following proposition is a straightforward consequence of Theorem \ref{ifqfor} and the isomorphism $H^*_K(X) \cong H_T^*(X)^W$.

\begin{prop}\label{sonbocneyoga}
If $X$ is equivariantly formal, then there is a natural isomorphism of graded $A^W$-algebras $$\phi: H_K^*(X) \cong H^0( \F_X)^W.$$
\end{prop}

\subsection{Examples}\label{exmplesf}

\subsubsection{Monodromy sheaves}\label{oercbulaocbe}\label{oercbulaocbe2}

So far our examples of a GKM-sheaf have come from $T$-spaces. Now we consider a construction of GKM-sheaves from combinatorial and algebraic data. Let $\Gamma$ be a GKM-\emph{graph} with a finite number of non-degenerate edges. Let $\E^{nd} \subset \E$ denote the set of nondegenerate edges. Choose an orientation for each $e \in \E^{nd}$, by defining source and target maps $s,t: \E^{nd} \rightarrow \V$, such that $I(e) = \{ s(e), t(e)\}$. A \emph{local system} on $\Gamma$ consists of a (finitely generated, $\Z$-graded) free $A$-module $M$ called the \emph{fibre} and a map $\rho: \E^{nd} \rightarrow Aut(M)$.

\begin{define}The \emph{monodromy GKM-sheaf} associated to the local system $(\Gamma, M, \rho)$ is the GKM-sheaf $\F = \F_{\rho}$ with stalks $\F(U_x) = M$ at vertices and degenerate edges, and $\F(U_e) = M \oplus M[2]$ at non-degenerate edges with restriction maps $ \F(U_e) = M \oplus M[2] \stackrel{res_e}{\rightarrow}  M\oplus M = \F(U_{s(e)}) \oplus \F(U_{t(e)})$ defined by the matrix 

\begin{equation}\label{matrex}
res_e :=	\left( \begin{array}{cc}
1 & \alpha(e)  \\
\rho(e) & -\alpha(e) \rho(e)  \end{array} \right)
\end{equation}
where we have abusively used $\alpha(e)$ to denote a generator of the projective weight $\alpha(e) \in \Pt$.
\end{define}

This construction produces a GKM-sheaf because the matrix (\ref{matrex}) becomes invertible after inverting $\alpha(e)$. 

\begin{example} The traditional GKM-manifolds of \cite{guillemin1998equivariant} provide the simplest example of a monodromy sheaf. In this case the fibre $M = A$ and all the automorphisms $\rho(e) \in Aut(A)$ are the identity.  
\end{example}

Monodromy sheaves were inspired by work of Guillemin-Holm \cite{gh} in the context of Hamiltonian actions on symplectic manifolds. A \emph{GKM manifold with non-isolated fixed points} is a closed $T$-manifold $X$, all of whose fixed point components are homeomorphic to a fixed reference space $F$, with GKM-graph $\Gamma_X = (\V,\E,I,\alpha)$ such that for each nondegenerate edge $E$ with (necessarily distinct) vertices $V_s$, $V_t$, there exists a commutative diagram of $T$-spaces:

\begin{equation*}\begin{CD}
	\xymatrix{ V_s \ar @/_/ [r]_{i_s} & \ar @/_/[l]_{\pi_s} E \ar @/^/ [r]^{\pi_t} &  \ar @/^/[l]^{i_t} V_t }
\end{CD}\end{equation*}
where $\pi_s$ and $\pi_t$ are $T$-equivariant $S^2$-fibre bundles for which the inclusions $i_s$ and $i_t$ are sections. 

\begin{prop}
If $\F_{\rho}$ is the monodromy GKM-sheaf on $\Gamma$ with fibre $H^*_T(F) \cong H_T^*(V)$ and $ \rho(E) = (\pi_s \circ i_t)^*$ then $\F_{\rho} \cong \F_X$.
\end{prop}
\begin{proof}
The isomorphisms $\F_{\rho}(U_V) \cong H_T^*(F) \cong \F_{X}(U_V)$ at vertices is clear from the definition. The isomorphism $$ \F_{\rho}(U_E) = H^*_T(F) \oplus H_T^*(F)[2] \cong H_T^*(E) = \F_X(U_E)$$ follows from the Thom isomorphism for the sphere bundle $\pi_s: E \rightarrow V_s\cong F$. That the restriction maps match up is an easy exercise.
\end{proof}

In \S \ref{rpvarforpsfut}, we show that for regular value $c \in K$, the representation variety $\Rp{c}{1}{K}$ is a GKM-manifold with nonisolated fixed points, and that $\F_{\Rp{c}{1}{K}}$ is a monodromy sheaf with non-trivial monodromy in general.

\subsubsection{Pure $\Gamma$-sheaves}

In \cite{braden2001mgi} Braden-MacPherson introduce the notion of a pure $\Gamma$-sheaf $\mathcal{M}$ over a moment graph. They also show that in many interesting cases, the equivariant intersection cohomology $IH_T^*(X)$ of a complex projective variety $X$ is equal to the global sections $H^0(\mathcal{M})$ of a pure $\Gamma$-sheaf $\M$ associated to $X$. 

\begin{prop}\label{gknbbcroeqp}
To any pure $\Gamma$-sheaf $\mathcal{M}$ there is a canonically associated GKM-sheaf $\mathcal{M}'$ such that $H^0(\mathcal{M}) \cong H^0(\mathcal{M}')$.	
\end{prop}

In our framework, a moment graph is essentially the same thing as a GKM-graph $(\V, \E, I, \alpha)$ with only a finite number of non-degenerate edges $\E^{nd}\subset \E$, and an ordering $I(e) = \{v_s, v_t\}$ for $e \in \E^{nd}$ subject to some further conditions. A \emph{$\Gamma$-sheaf} $\mathcal{M}$ consists of $A$-modules $\mathcal{M}(v)$ for each vertex and $\mathcal{M}(e)$ for each $e \in \E^{nd}$ and a homomorphism $\rho_{v,e}: \mathcal{M}(v) \rightarrow \mathcal{M}(e)$ for every pair $\{(v,e) \in \V \times \E^{nd}| v \in I(e)\}$. The module of global sections is 
\begin{equation}\label{orceburcerasidoe}
	H^0(\M) = \{ (m_x) \in \bigoplus_{x \in \V \cup \E^{nd}} \mathcal{M}(x) |~\rho_{v,e}(m_v) = m_e ~\forall \rho_{v,e}\}\end{equation}

\begin{rmk}\label{nkbnoitang;n}
It is explained in \cite{braden2001mgi} that the data defining a $\Gamma$-sheaf is equivalent to a sheaf of $A$-modules over $\V \cup \E^{nd}$ in the topology dual to the one we use, in the sense that their closed sets are equal to our open sets and vice-versa.
\end{rmk}

A $\Gamma$-sheaf is called \emph{pure} if $\M(v)$ is free for all $v \in \V$ and if for every edge $e\in \E^{nd}$ with $I(e) = \{v_s, v_t\}$ satisfies, $\M(e) \cong \M(v_s)/ \alpha(e)\M(v_s)$ with $$\rho_{v_s,e}: \M(v_s) \rightarrow \M(e) = \M(v_s)/\alpha(e)\M(v_s)$$ the projection, plus some additional conditions.

\begin{proof}[Proof of Proposition \ref{gknbbcroeqp}]
Let $\M$ be a pure $\Gamma$-sheaf. Define a new sheaf of $A$-modules $\M'$ by
\begin{itemize}

\item $\M'(\{v\}) = \M'(U_v) = \M(v)$ for all $v \in \V$, 

\item for all $e \in \E^{nd}$ the stalk $\M'(U_e)$ and restriction map $res_e$ are defined by the short exact sequence 

\begin{equation}\label{gncngoawpqpnpbg}
0 \rightarrow \M'(U_e) \stackrel{res_e}{\rightarrow} \M(v_s) \oplus \M(v_t) \stackrel{\rho_{v_s,e} - \rho_{v_t,e}}{\longrightarrow} \M(e) \rightarrow 0,
\end{equation}

\item for all degenerate edges $e \in \E \setminus \E^{nd}$,  $\M'(e) = \M'(I(e))$ with $res_e$ the identity map.

\end{itemize}
Clearly $\M'$ is a sheaf and $res_e$ is an isomorphism modulo $\alpha(e)$ for all $e\in \E$. Both $H^0(\M)$ and $H^0(\M')$ inject into $\bigoplus_{v \in \V} \M(v)$ by projection, and the short exact sequence (\ref{gncngoawpqpnpbg}) ensures that they have the same image, so $H^0(\M) \cong H^0(\M')$.

It remains to show that $\M'(U_e)$ is a free $A$-module for every $e \in \E^{nd}$. Because $\M(v_s)$ and $\M(v_t)$ are projective, $\rho_{v_s,e} - \rho_{v_t,e}$ lifts through (\ref{orceburcerasidoe}) to a surjective map $f: \M(v_s) \oplus \M(v_t) \rightarrow \M(v_s)$. Because $f$ admits a section, $\ker(f) \cong \M(v_t)$ and $\M'(U_e)$ fits into an exact sequence $0 \rightarrow \ker(f) \rightarrow \M'(U_e) \rightarrow \alpha(e)\M(v_s) \rightarrow 0$ which must also split so $$\M'(U_e) \cong \M(v_t) \oplus \alpha(e)\M(v_s)$$ is free.  
\end{proof}

\subsubsection{Mutants of compactified representations}

These are examples of closed $T$-manifolds introduced by Franz and Puppe \cite{franz2008freeness} whose equivariant cohomology is torsion-free but not free over $A$. There are three examples: $Z_r$, $r=2,4,8$ with action by $T = (U(1))^{r+1}$, whose construction makes use of the Hopf fibration $S^{r-1} \rightarrow S^{2r-1} \rightarrow S^{r}$ . Franz and Puppe prove that as graded $A$-modules 
$$ H_T^*(Z_r) \cong A \oplus m[r-1] \oplus A[2r+2] \oplus A[3r+1] $$  
where $m$ is the augmentation ideal of $A$. 

The GKM-sheaf $\F_{Z_r}$ is easily determined from the description of $Z_r$ in \cite{franz2008freeness}. The GKM-hypergraph $\Gamma_{Z_r}$ consists of two vertices and $r+1$ non-degenerate edges labelled by distinct weights $\alpha_0, ..., \alpha_r$ which form a basis of $\lie{t}^*$. The GKM-sheaf $\F_{Z_r}$ is the trivial monodromy sheaf with fibres $M := H^*(S^{r-1})\otimes A \cong A \oplus A[r-1]$. It follows that the global sections of $\F_{Z_r}$ may be identified with the image of the matrix
$$	\left( \begin{array}{cc}
1 & f   \\
1 & -f   \end{array} \right)$$
in $ M \oplus M$, where $f = \prod_{i=0}^r \alpha_i$ has degree $2(r+1)$. Thus as graded $A$-modules

$$H^0(\F_{Z_r}) \cong A \oplus A[r-1] \oplus A[2r+2] \oplus A[3r+1].$$
This example shows that $H^0(\F_{X})$ may be a free $A$-module even when $H^*_T(X)$ is not.

\subsubsection{Equivariant de Rham theory and graphs}
Given a GKM-graph $\Gamma$, consider the trivial monodromy sheaf $\F_{Id}$ over $\Gamma$ with fibre $A$. The \emph{graph cohomology} $H_T(\Gamma)$ is defined to equal $H^0(\F_{id})$. It is interesting to ask under what circumstances $H_T(\Gamma)$ is a free $A$-module and what its Betti numbers are. 

In a series of papers, \cite{guillemin1998equivariant,guillemin20011,guillemin2003egf} Guillemin and Zara translate concepts from Hamiltonian actions on symplectic manifolds to GKM-theory, motivated in part by these questions. They a define the notion of moment map on a GKM-graph $\Gamma$, define the reduction $\Gamma// S^1$ with respect to such a moment map and prove a version of Kirwan surjectivity $\kappa: H_T(\Gamma) \rightarrow H_{T/S^1}(\Gamma//S^1)$.

In certain circumstances, the reduction $\Gamma//S^1$ is not a graph but a GKM-\emph{hyper}graph and in this situation, addressed in \cite{guillemin2003egf}, the arguments become rather technical. We believe the reduction process could be more clearly understood in our framework (we do not pursue this in this paper). For example, the definition of $H_T(\Gamma//S^1)$  ``by duality'' is a strong hint that our topology, which is dual to the more obvious one in Braden-MacPherson \cite{braden2001mgi} may be important (see Remark \ref{nkbnoitang;n}). Also, the locality of $H_T(\Gamma//S^1)$ would be manifest if it were defined as the global sections of a GKM-sheaf.

\section{Representation varieties for a punctured surface}\label{rpvarforpsfut}

It will be useful to introduce a larger class of representation varieties than those defined in the introduction. Gauge theoretically, these varieties correspond to moduli spaces of flat bundles over a punctured nonorientable surface with prescribed holonomy around the puncture. 

Let $\Sigma_g$ denote the connected sum of $g+1$ copies\footnote{The index is chosen so that $g$ is the genus of the orientable double cover.} of $\R P^2$. By the classification of compact surfaces, every nonorientable compact surface without boundary is isomorphic to $\Sigma_g$ for some $g =0,1,2,...$. The fundamental group $\pi_1(\Sigma_g)$ has presentation

\begin{equation}\label{cofcaff}
 \pi_1(\Sigma_g) \cong \{ a_0,...,a_g| \prod_{i=0}^g a_i^2 = \id\}. 
\end{equation}
Let $\Gamma_g$ denote the free group on $g+1$ generators $\{a_0,...,a_g\}$ so that the presentation (\ref{cofcaff}) determines a surjection $\Gamma_g \rightarrow \pi_1(\Sigma_g)$. For $K$ a compact, connected Lie group of rank $r$ and $c \in K$, define
 $$\Rp{c}{g}{K}:= \{\phi \in Hom(\Gamma_g,K) | ~\phi(\prod_{i=0}^g a_i^2) = c\}. $$ The embedding $ \Rp{c}{g}{K} \hookrightarrow K^{g+1}$ sending $\phi$ to $(\phi(a_0),...,\phi(a_g))$ identifies $\Rp{c}{g}{K}$ with the compact real algebraic set 
\begin{equation}\label{gnognognn;cmto}
\Rp{c}{g}{K} \cong \{ (k_0,...,k_g)\in K^{g+1}| \prod_{i=0}^g k_i^2 = c\}
\end{equation}
on which the centralizer $Z_K(c) \subseteq K$ acts by conjugation.
 Notice that $$\Rp{\id}{g}{K} \cong Hom(\pi_1(\Sigma_g),K)$$ where $\id \in K$ is the identity, so we recover the representation varieties described in the introduction.

We will always choose $c$ to lie in a fixed maximal torus $T \subset K$, with complexified lie algebra $\lie{t}$. We call $c$ \emph{regular} if $Z_K(c) =T$. We use notation $W := N_K(T)/T$ for the Weyl group, and $W_c := N_{Z_K(c)}(T)/T$ for the portion of the Weyl group that stabilizes $c$. The isomorphism $$ H_{Z_K(c)}^*(\Rp{c}{g}{K})  \cong H_T^*(\Rp{c}{g}{K})^{W_c}$$ allows us to divide the study of general compact group actions into torus actions and Weyl group actions. As before, we use notation $A := \C[\lie{t}] = S(\lie{t}^*) \cong H^*(BT)$.

\subsection{Main results}\label{mainresults}

In this section we construct the GKM-sheaf of the representation variety $\Rp{c}{g}{K}$ using a monodromy sheaf, convolution products and induction.

Let $T_2$ denote the subgroup of $2$-torsion elements of $T$.  That is, $$T_2 := \{ t \in T| t^2 = \id\}.$$ Then $T_2$ is isomorphic to the finite group $(\Z/2)^r$.  Let $\V$ be a torsor for $T_2$, meaning that $\V$ is a finite set of order $2^r$ equipped with a free and transitive action $T_2 \times \V \rightarrow \V$. Let $\Delta \subset \Pt$ denote the set of roots of $T \subset K$ modulo non-zero scalar multiplication. For each $\alpha \in \Delta$, there exists a pair of  co-roots $\pm h_{\alpha} \in \lie{t}$ and an element $exp( \pm i \pi h_{\alpha}) \in T_2$ (observe that the sign ambiguity disappears after exponentiating). For each $\alpha \in \Pt$, define the equivalence relation $\sim_{\alpha}$ on $\V$ to be trivial if $\alpha \not\in \Delta$ and to be generated by the relation $$ v \sim_{\alpha} exp(i\pi h_{\alpha}) v$$ for all $v \in \V$, if $\alpha \in \Delta$. Together $(\V, \sim)$ form a GKM-graph which we denote $\Gamma(K)$.  Our first result is

\begin{thm}\label{indepofgraph}\label{thma}
For $g\geq 1$, the isomorphism type of the GKM-hypergraph $\Gamma_{\Rp{c}{g}{K}}$ is independent of both $g$ and $c$ and is equal to $\Gamma(K)$.
\end{thm}

Observe next that the $T_2$-action on $\V$ respects the equivalence relations $v \sim_{\alpha} exp(i\pi h_{\alpha}) v$ because $T_2$ is abelian.  Thus we may consider the category $GKM_{T_2}(\Gamma(K))$ of $T_2$-equivariant GKM-sheaves on $\Gamma(K)$ and this category acquires a convolution product $*$ by \S \ref{aocegbulagcobj}.

\begin{thm}\label{beorbaorecbia}\label{thmb}\label{bbnxnrpin}
Let $c \in T \subseteq K$ be a regular element. Then for any $g\geq 1$, the GKM-sheaf $\F_{\Rp{c}{g}{K}}$ is $T_2$-equivariant and there is an isomorphism
$$ \F_{\Rp{c}{g}{K}} \cong \F_{\Rp{c}{1}{K}} * ... * \F_{\Rp{c}{1}{K}} $$
with the $g$-fold convolution product of $\F_{\Rp{c}{1}{K}}$.
\end{thm}

Suppose now that $c \in T \subseteq K$ is not necessarily regular.

\begin{thm}\label{oalboecrlbitnoeh}\label{thmc}\label{bnnsopgnrpan}
	Let $c, c' \in T \subseteq K$ and suppose that $c'$ is regular. There exist twisted actions of $W_{c}$ on both $\F_{\Rp{c}{g}{K}}$ and $\F_{\Rp{c'}{g}{K}}$ such that
	$$ H^0(\F_{\Rp{c}{g}{K}})^{W_{c}} \cong H^0(\F_{\Rp{c'}{g}{K}})^{W_{c}} .$$ 
\end{thm}
In Theorem \ref{oalboecrlbitnoeh}, the twisted action of $W_{c}$ on $\F_{\Rp{c}{g}{K}}$ is induced by the topological action of $Z_K(c)$ on the representation variety $\Rp{c}{g}{K}$ in the standard way (see Proposition \ref{sonbocneyoga}). The twisted action of $W_{c}$ on $\F_{\Rp{c'}{g}{K}}$ is more subtle, and is not induced from an explicit action on $\Rp{c'}{g}{K}$.

Theorems \ref{beorbaorecbia} and \ref{oalboecrlbitnoeh} reduce the problem of constructing the GKM-sheaf $\F_{\Rp{c}{g}{K}}$ to the special case when $g=1$ and $c$ is regular. In this case $\F_{\Rp{c}{1}{K}}$ is isomorphic to a monodromy GKM-sheaf over $\Gamma(K)$, supposing a mild additional condition on $K$, as we now explain.  Let $$M = \wedge(\lie{t}^*)\otimes_{\C} A$$ with $\wedge (\lie{t}^*)$ graded so that $\lie{t}^*$ has degree one. Each $\alpha \in \Delta$ determines a Weyl group reflection on $\lie{t}$ and we denote by $S_{\alpha}$ the induced action on $\wedge(\lie{t}^*)$. For a given non-degenerate edge $e$ of $\Gamma(K)$ with $\alpha(e) \in \Delta$, we define $$\rho(e) = S_{\alpha(e)}\otimes_{\C} Id_A \in Aut(M).$$ 
Observe that $\rho(e) \circ \rho(e)$ is the identity on $M$, so it is unnecessary to specify the orientation of the edges.

\begin{thm}\label{thmd}
Let $c \in T \subseteq K$ be regular and suppose that $\pi_1(K)$ is $2$-torsion free (i.e. only the identity element squares to the identity). Then there is an isomorphism between the GKM-sheaf $\F_{\Rp{c}{1}{K}}$ and the monodromy sheaf $\F_{\rho}$ with data $M$ and $\rho$ described above.
\end{thm}

Finally, suppose that $\pi_1(K)$ contains $2$-torsion.  Choose a finite, connected covering group $\phi: \tilde{K}\rightarrow K$ such that $\pi_1(\ti{K})$ is $2$-torsion free. Let $\ti{T} \subseteq \ti{K}$ and $T \subset K$ be maximal tori with $\phi(\ti{T})=T$. Let $\ti{c}$ and $c$ be respective regular elements and let $\ti{T}_2$ and $T_2$ be respective $2$-torsion subgroups. Restricting $\phi$ to the homomorphism $\ti{T}_2 \rightarrow T_2$ induces an induction functor $$Ind_{\ti{T}_2}^{T_2}: GKM_{\ti{T}_2}(\Gamma(\ti{K})) \mapsto GKM_{T_2}(\Gamma(K)). $$

\begin{thm}\label{thme}
If $c$ is regular, then there is an isomorphism of $T_2$-equivariant GKM-sheaves
$$ \F_{\Rp{c}{g}{K}} \cong Ind_{\ti{T}_2}^{T_2}( \F_{\Rp{\ti{c}}{g}{\ti{K}}}). $$
\end{thm}

\subsection{Fixed points}

Our first task is to describe the fixed points of $\Rp{c}{g}{K}$ under the conjugation action by $T$. Maximal tori are maximal abelian, so it follows that a homomorphism $\phi \in \Rp{c}{g}{K}$ is fixed by $T$ if and only if $\im( \phi) \subset T$, thus  
$$ \Rp{c}{g}{K}^T = \Rp{c}{g}{T}.$$
For this reason, it is useful to describe with some care the case that $K = T = U(1)^r$ is a torus.  

Any homomorphism from $\pi := \pi_1(\Sigma_g)$ to an abelian group must factor through the abelianization $ \pi/[\pi,\pi] \cong H_1( \Sigma_g, \Z) \cong \Z^g \oplus \Z_2$ so we obtain

\begin{equation}\label{chekout}
 \Rp{\id}{g}{T} = Hom(\pi,T) \cong Hom( H_1( \Sigma_g, \Z), T) \cong T^g \times T_2.
\end{equation}
where $T_2  = Tor_1^{\Z}(\Z/2,T)\cong (\Z/2)^r$ is the 2-torsion subgroup of $T$. More canonically, regarding $\Rp{\id}{g}{T} $ as a group under pointwise multiplication, we have a short exact sequence

\begin{equation}\begin{CD}\label{oelcbaurcbd}
	\xymatrix{0 \ar[r] & T^g \ar[r] & \ar @/_/[l]_{\rho} \Rp{\id}{g}{T} \ar[r] & T_2 \ar[r] &  0}
\end{CD}\end{equation}
and we choose a splitting $\rho: \Rp{\id}{g}{T} \rightarrow T^g$, $\rho(\phi) = (\phi(a_1),...,\phi(a_g))$, using the presentation (\ref{cofcaff}) of the fundamental group.

For general $c \in T$, pointwise multiplication determines a free and transitive action of $\Rp{\id}{g}{T} $ on $\Rp{c}{g}{T}$, i.e. 

\begin{lem}\label{twotoris}
For $g\geq 0$, $T$ a rank $r$ compact torus and $c \in T$, the representation variety $\Rp{c}{g}{T}$ is a torsor for $ \Rp{\id}{g}{T}$, thus is diffeomorphic to $T^g \times T_2$.
\end{lem}

\begin{rmk}\label{fxpts}
The action of $T_2$ on $\Rp{c}{g}{T}$ is canonically defined up to isotopy and acts freely and transitively on the set of components. Thus we have a canonical isomorphism $H_T^*(F_i) \cong H_T^*(F_j)$ for every pair of connected components $F_i, F_j$ of $\Rp{c}{g}{T}$. Less canonically, these cohomology rings may all be identified with $ H_T(T^g) \cong (\wedge \clie{t}^*)^{\otimes g} \otimes_{\C} A$, where $A := \C[\lie{t}]$ as before.
\end{rmk}

\begin{cor}\label{toreghe}
For any compact, connected Lie group $K$ of rank $r$, the GKM-hypergraph of $\Rp{c}{g}{K}$ has a vertex set $\V$ consisting of $2^{r}$ vertices which are naturally identified with the set of square roots of $c$, $$\V \cong \{ t \in T| t^2 =c\}$$
upon which $T_2$ acts freely and transitively by multiplication. The stalks of the GKM-sheaf $\F_{\Rp{c}{g}{K}}$ at vertices may be identified with $ \wedge(\lie{t}^*)^{\otimes g} \otimes A$ and the $T_2$-action preserves these identifications. 
\end{cor}
Observe that Corollary \ref{toreghe} implies that the vertex set of $\Gamma_{\Rp{c}{g}{K}}$ is independent of both $c$ and $g$, in keeping with Theorem \ref{indepofgraph}. 

Because $T_2$ acts freely on the vertex set $\V$, we gain a convolution product $*$ on $T_2$-equivariant GKM-sheaves (defined in \S \ref{aocegbulagcobj}).

\begin{lem}\label{insapcing}
For $g\geq 0$ and $c \in T$ let $\F_g := \F_{\Rp{c}{g}{T}}$. Then $$ \F_g  \cong \F_1 * ... * \F_1 = (\F_1)^{* g}$$
\end{lem}

\begin{proof}
We begin with the case $c = \id$. Use the characterization (\ref{gnognognn;cmto}) $$\Rp{\id}{g}{T} = \{ (t_0,...,t_g) \in T^{g+1}| t_0^2...t_g^2 = \id\}$$ with $T_2$ acting by multiplying $t_0$. There is a $T_2$-equivariant homeomorphism 
$$ \psi: \Rp{\id}{1}{T} \times_{T_2} \Rp{\id}{g}{T} \cong \Rp{\id}{g+1}{T}$$
where $ \psi( (s_0, s_1) \times (t_0,...,t_g)) = (s_0 t_0, s_1, t_1,t_2,...,t_g)$. Applying Proposition \ref{clogbairdisagooddog} and induction completes the proof. The case of general $c \in T$ follows by the fact that $\Rp{c}{g}{T}$ is a torsor for $\Rp{\id}{g}{T}$, thus is $T_2$-equivariantly diffeomorphic.
\end{proof}

\begin{rmk}\label{entries}
In the proof of Lemma \ref{insapcing}, $T_2$ acts on $\Rp{\id}{g}{T}$ by multiplying the $0$th entry $t_0$. We could have chosen instead to multiply one of the other entries and the action would be the same up to isotopy.
\end{rmk}

For $c \not= \id$, the isomorphism $\F_g \cong \F_1 *...*\F_1$ of Lemma \ref{insapcing} has only been defined up to an automorphism of $\F_g$ induced by $T_2$, or equivalently up to a choice of base vertex. The following Lemma is meant to address this ambiguity. 

\begin{lem}\label{corwhat}
A continuous path $\gamma: [0,1] \rightarrow T$ induces an isotopy class of $T_2$-equivariant homeomorphisms $ \Rp{\gamma(0)}{g}{T} \stackrel{~}{\rightarrow} \Rp{\gamma(1)}{g}{T}$ (where $T_2$ acts as in Lemma \ref{twotoris}). This class contains a map sending $(t_0,....,t_g)$ to $(st_0,t_1,...,t_g)$, where $s \in T$ satisfies $s^2 = \gamma(0)^{-1}\gamma(1)$.
\end{lem}

\begin{proof}
The variety $\Rp{c}{g}{T}$ is the fibre at $c\in T$ of the submersion $ T^{g+1} \rightarrow T$ sending $(t_0,...,t_g)$ to $\prod_{i=0}^g t_i^2$. Pulling back by $\gamma$ determines a fibre bundle $\gamma^* T^{g+1}$ over $[0,1]$ inducing a homeomorphism up to isotopy between fibres $ \Rp{\gamma(0)}{g}{T}$ and $ \Rp{\gamma(1)}{g}{T}$. 

The $T_2$-action is induced by a canonical $\Rp{\id}{g}{T}$-action. That the homeomorphism can be made equivariant follows from the classical result of Palais-Stewart \cite{ps} on the rigidity of compact group actions on compact manifolds. 

Define the path $\delta: [0,1] \rightarrow T$ by $\delta(x) = \gamma(0)^{-1}\gamma(x)$. Because the squaring map is a covering, there is a unique path $\sqrt{\delta}: [0,1] \rightarrow T$ such that $(\sqrt{\delta}(x))^2 = \delta(x)$ and $\sqrt{\delta}(0) = \id$. Define a $T_2$-equivariant bundle trivialization  $\phi: [0,1] \times  \Rp{\gamma(0)}{g}{T} \rightarrow \gamma^* T^{g+1}$ by $ \phi(x, (t_0,...,t_g)) = (x, (\sqrt{\delta}(x)t_0,....,t_g)) $. Setting $s = \sqrt{\delta}(1)$ completes the proof.
\end{proof}

\begin{cor}\label{lbloeblrbi}
Let $\gamma: [0,1] \rightarrow T$ be a continuous path and denote $\mathcal{R}_i = \Rp{\gamma(i)}{g}{K}$ for $i=0,1$. Then $\gamma$ determines a $T_2$-equivariant bijection between the vertex sets $\V_0 \cong \V_1$ that lifts to a $T_2$-equivariant isomorphism of restricted sheaves $$ \F_{\mathcal{R}_0}|_{\V_0} \cong \F_{\mathcal{R}_1}|_{\V_1}.$$
\end{cor}

\subsection{The case K=SU(2)}\label{case su2}

We review here the main results from \cite{baird2008msf}, using the language of GKM-sheaves. Throughout \S \ref{case su2}, set $K= SU(2)$, and $T \subset K$ is a maximal torus with $c \in T$.  The centre $Z(K)$ consists of $\pm \id = T_2$, and all other values of $c$ are regular. 

\begin{thm}[ \cite{baird2008msf} Thm 1.2 ]
The representation varieties $\Rp{c}{g}{K}$ are equivariantly formal under conjugation by $T$ for all $c \in K$ and $g \in \{0,1,2,...\}$. This means in particular that $H_T^*(\Rp{c}{g}{K}) \cong H^0(\F_{\Rp{c}{g}{K}})$ is free over $A$.
\end{thm}

Since $T$ has rank one, $\Pt$ is a single point. 

\begin{prop}
The GKM-hypergraph $\Gamma_{\Rp{c}{g}{K}}$ for any $c \in T$ and $g\geq 1$  consists of two vertices and a single edge $e$ connecting them, with $\alpha(e)$ equal to the sole element of $\Pt$. 
\end{prop}

\begin{proof}
This follows from Corollary \ref{toreghe} and the fact the $\Rp{c}{g}{K}$ is connected.	
\end{proof}

Because $\Gamma_{\Rp{c}{g}{K}}$ has only a single edge $e$, the GKM-sheaf $\F_{\Rp{c}{g}{K}}$ is completely determined by the localization map $$H_T^*(\Rp{c}{g}{K}) = \F_{\Rp{c}{g}{K}}(U_e) \rightarrow  H_T^*(\Rp{c}{g}{T})=\F_{\Rp{c}{g}{K}}(I(e)). $$

The case when $g=1$ and $c$ is regular can be described quite explicitly. In this case $\Rp{c}{1}{K}$ is diffeomorphic to $S^1 \times S^2$ and $T$ acts via rotation on $S^2$ with weight $2$. In particular (recall the definitions of GKM-manifold with non-isolated coefficients and monodromy sheaf from \S \ref{oercbulaocbe}):

\begin{prop}\label{jydiv}
For regular $c$, $\Rp{c}{1}{K}$ is a GKM-manifold with non-isolated fixed points (\S \ref{oercbulaocbe2}). The GKM-sheaf $\F_{\Rp{c}{1}{K}}$ is isomorphic to a monodromy sheaf with fibre $ H^*_T(T) \cong \wedge \lie{t}^* \otimes A$, with holonomy map $\rho(e) = S_{\alpha(e)} \otimes Id_{A}$, where $S_{\alpha(e)}$ is the automorphism of $\wedge(\clie{t}^*)$ induced by multiplication by $-1$ on $\lie{t}$, which we think of as reflection in the root hyperplane $\alpha(e)^{\perp} = \{0\}$.
\end{prop}

\begin{proof}
That $\Rp{c}{1}{K}$ is a GKM-manifold with non-isolated fixed points is clear from the preceding description. The fibres were described in Remark \ref{fxpts}. The holonomy map can be inferred from Propositions 5.3 and 5.4 of \cite{baird2008msf}. 
\end{proof}

\begin{rmk}
Notice that because the graph $\Gamma_{\Rp{c}{1}{K}}$ is ``simply connected", the local system above can be trivialized. However the convention adopted in Remark \ref{fxpts} of using the $T_2$-action to identify fixed point components forces the local system to be non-trivial. Furthermore, when we consider higher rank groups in \ref{sect82}, the analogous local system will not be trivializable in general.
\end{rmk}

Because $T_2 = \{ \pm \id\}$ lies in the centre of $K$, the $T_2$-action on $\Rp{c}{g}{T}$ described in Remark \ref{fxpts} extends to $\Rp{c}{g}{K}$ in the obvious way. This makes $\F_{\Rp{c}{g}{K}}$ into a $T_2$-equivariant GKM-sheaf.

\begin{prop}\label{dhiontoijw}
Let $ c \in T$ be a regular element and let $\F_g := \F_{\Rp{c}{g}{K}}$. For $g\geq 1$ we have an isomorphism between the GKM-sheaves $\F_{g} \cong  \F_1 * ... *\F_1 = (\F_1)^{* g}$. 
\end{prop}

\begin{proof}
Follows from Propositions 5.3 and 5.4 of \cite{baird2008msf}.
\end{proof}

\begin{rmk}
In view of Proposition \ref{touseorbit}, one might suspect that Proposition \ref{dhiontoijw} is the consequence of a homeomorphism or at least a $T$-map between $\F_{\Rp{c}{g}{K}}$ and $\F_{\Rp{c}{g}{K}}\times_{T_2} ....\times_{T_2} \F_{\Rp{c}{g}{K}}$. As explained in \cite{baird2008msf}, the real story is trickier than this. 
\end{rmk}

Now we turn to the non-regular cases $\epsilon \in \{\pm \id\}$. The full group $K$ centralizes $\epsilon$, so $K$ acts by conjugation on $\Rp{\epsilon}{g}{K}$, and $\F_{\Rp{\epsilon}{g}{K}}$ acquires a twisted $W_{\epsilon} = W$-action as described in \S \ref{twsithgadctiongs}. Choose a path $\gamma: [0,1] \rightarrow K$ connecting $\epsilon$ with some regular element $c$. Using Lemma \ref{corwhat} we obtain a twisted action of $W$ on the restriction of $\F_{\Rp{c}{g}{K}}$ to the vertex set.

\begin{prop}\label{supbut}
The action described above extends to a twisted action of the Weyl group $W$ on $\F_{\Rp{c}{g}{K}}$, twisted by the standard action of $W$ on $\lie{t}$. Taking $W$-invariants produces an isomorphism $$ H_{K}^*\big(\Rp{\epsilon}{g}{K}\big)  \cong H^0(\F_{\Rp{c}{g}{K}})^W.$$  
\end{prop}

\begin{proof}
Use (\ref{oelcbaurcbd}) to identify $\Rp{\id}{g}{T} \cong T^g \times T_2$. By Lemma \ref{corwhat}, any path connecting $\id$ to $-\id$ also determines an isomorphism $\Rp{-\id}{g}{T} \cong T^g \times T_2$. In terms of these isomorphisms, the action of the nontrivial element $w \in W$, sends $(t_1,...,t_g, z)$ to $(wt_1w, ...., w t_g w, \epsilon z)$. The result now follows from the explicit description of the image of the localization map $H_T^*(\Rp{\epsilon}{g}{K}) \rightarrow H_T^*(\Rp{\epsilon}{g}{T})$ found in Propositions 5.3, 5.4 and 5.5 from \cite{baird2008msf}.
\end{proof}

\begin{rmk}This result is stranger than it might first appear. The $W$-action on $\F_{\Rp{c}{g}{K}}$ is \emph{not} in general induced by one on $\Rp{c}{g}{K}$. Moreover, while the result implies that $H_{T}^*\big(\Rp{\epsilon}{g}{K}\big)^W  \cong H^0(\F_{\Rp{c}{g}{K}})^W $, it is \emph{not} true in general that $H_{T}^*\big(\Rp{\epsilon}{g}{K}\big) \cong H^0(\F_{\Rp{c}{g}{K}})$.
\end{rmk}

\begin{rmk}\label{for later}
It will be important later to observe that for both $\epsilon \in \{ \pm \id\}$, the action of $W$ on $\F_{\Rp{c}{g}{K}}$ described in Proposition \ref{supbut} commutes with the $T_2$-action, and these actions differ by the non-trivial element of $T_2$. 
\end{rmk}

\subsection{$K$ has semisimple rank one}\label{semisimple rank one}

Let $K$ be a compact connected Lie group of rank $r$ with complexified Lie algebra $\lie{k}$ and let $T \subset K$ be a maximal torus with complexified Lie algebra $\lie{t}$. We may decompose $\lie{k}$ into its central and semisimple parts:

\begin{equation*}
\lie{k} = Z(\lie{k}) \oplus \lie{k}_{ss}
\end{equation*}
In this section, we consider the case where $\lie{k}_{ss}$ has rank $1$ or equivalently, $\lie{k}_{ss} \cong \lie{su}(2)\otimes \C$. The following lemma is elementary.

\begin{lem}\label{three}
If $K$ has semisimple rank one, then it is isomorphic to one of the following:\\
(i) $U(1)^{r-1} \times SU(2)$\\
(ii) $U(1)^{r-2}\times U(2)$\\
(iii) $U(1)^{r-1} \times SO(3)$
\end{lem}

If $K = SU(2) \times U(1)^{r-1}$ and $c = (c_1,c_2)$, then the representation variety factors as a product spaces already described 

\begin{equation}\label{aolerucdlaorced}
\Rp{c}{g}{K} = \Rp{c_1}{g}{SU(2)} \times \Rp{c_2}{g}{U(1)^{r-1}}.
\end{equation}

If $K$ has form (ii) or (iii), then it fits into a short exact sequence

\begin{equation}\label{shorex}
 0 \rightarrow C_2 \rightarrow \ti{K} \stackrel{\phi}{\rightarrow} K \rightarrow 1,
\end{equation}
where $\ti{K}=U(1)^{r-1} \times SU(2)$  and $ C_2 \cong \Z/2$ is central in $\ti{K}$. Thus 
\begin{equation}\label{galgoidsocov}
\Rp{c}{g}{K} \cong \coprod_{\phi(\ti{c}) = c} \Rp{\ti{c}}{g}{\ti{K}}/ C_2^{g+1}.
\end{equation}
where $C_2^{g+1}$ acts by multiplying $g+1$-tuples entry-wise.

\begin{prop}\label{darankone}	
For all $g \geq 0$ and $c \in T$ the representation variety $\Rpe{} := \Rp{c}{g}{K}$ is equivariantly formal, so $H^0(\F_{\Rpe{}}) \cong H_T^*(\Rpe{})$ by Theorem \ref{ifqfor}.
\end{prop}

\begin{proof}
If $K$ is type (i) then the result follows from (\ref{aolerucdlaorced}) because the product of equivariantly formal spaces is equivariantly formal under the product group action. 

In the remaining cases, apply equivariant cohomology to (\ref{galgoidsocov})
$$H_T^*(\Rp{c}{g}{K}) \cong H_{\ti{T}}^*(\Rp{c}{g}{K}) \cong \bigoplus_{\phi(\ti{c}) = c}H_{\ti{T}}^*(\Rp{\ti{c}}{g}{\ti{K}})^{C_2^{g+1}}$$
so $H_T^*(\Rp{c}{g}{K}) $ is a summand of a free $A$-module, hence free.
\end{proof}

\begin{rmk}There is a small subtlety in the above proof that should be explained. We are working with complex coefficients so the homomorphism $\ti{T} \rightarrow T$ induces an isomorphism $H_{\ti{T}}^*(\Rp{\ti{c}}{g}{\ti{K}}) \cong H_T^*(\Rp{\ti{c}}{g}{\ti{K}})$. Thus formality with respect to $T$ is equivalent to formality with respect to $\ti{T}$. 
\end{rmk}

\begin{rmk}\label{lracolercuh}
Since $C_2$ is a subset of the two torsion group $\ti{T}_2 \subset \ti{T}\subseteq \ti{K}$, it follows from Remark \ref{entries} each factor of $C_2^{g+1}$ acts the same way on $\Rp{\ti{c}}{g}{\ti{T}})$ up to homotopy.  Thus the GKM-sheaf of $\Rp{c}{g}{K}$ coincides with the GKM-sheaf of $\Rp{\ti{c}}{g}{\ti{K}}/C_2$ where $C_2$ acts by multiplying (say) the $0$th entry.
\end{rmk}

We turn to studying the GKM-sheaf $\F_{\Rp{c}{g}{K}}$. We begin with determining the GKM-hypergraph. Observe that since $K$ has semisimple part of rank one, the pair $(K, T)$ has a unique pair of roots, which we denote $\pm \alpha$. Every point of $\Rp{c}{g}{K}$ is fixed by $\ker(\alpha)$, so the one-skeleton is the whole space $(\Rp{c}{g}{K}, \Rp{c}{g}{T})$. 

Define the co-root $h_{\alpha} \in \lie{t}$ to be the unique element in $\lie{t} \cap \lie{k}_{ss}$ satisfying $\alpha(h_{\alpha}) = 2$. The exponential $exp(2\pi i h_{\alpha}) = \id \in K$, so $exp(\pi i h_{\alpha}) \in T_2$ .

\begin{prop}\label{bbboty}
For any $c\in T$ and $g\geq 1$, denote the GKM-hypergraph $\Gamma_{\Rp{c}{g}{K}} = (\V,\sim)$. The vertex set $\V$ is naturally identified with the $T_2$-torsor
\begin{equation*}
\V \cong \{ t \in T| t^2 = c\}.
\end{equation*}
 The equivalence relation $\sim_{\alpha}$ is discrete unless $\alpha$ is the root, in which case
 \begin{equation}\label{piears}
 v \sim_{\alpha} exp(\pi i h_{\alpha})\cdot v .
\end{equation}
 for all $v \in \V$. In particular $\Gamma_{\Rp{c}{g}{K}} =\Gamma(K)$.
\end{prop}

\begin{proof}
The vertex set was explained in Corollary \ref{toreghe}. If $K \cong \ti{K} \cong SU(2)\times U(1)^{r-1}$, then $\Rp{c}{g}{K} = \Rp{c_1}{g}{SU(2)} \times \Rp{c_2}{g}{U(1)^{r-1}}$ and the result follows easily. For the remaining cases we use Remark \ref{lracolercuh} to identify $\Gamma_{\Rp{c}{g}{K}}$ with the quotient of the $C_2$ action on $\coprod_{\phi(\ti{c}) = c} \Gamma_{\Rp{\ti{c}}{g}{\ti{K}}}$ and the result follows.
\end{proof}

\begin{rmk}
The element $exp(\pi i h_{\alpha}) \in T_2$ is non-trivial if $K$ has type (i) or (ii) and is the identity when $K$ has type (iii). Thus $\Gamma_{\Rp{c}{g}{K}}$ is a bipartite graph in the first two cases and is discrete in the third.
\end{rmk}

\subsubsection{If $c$ is regular}\label{kgibiagbi}

As explained in \S \ref{mainresults}, the GKM-graph $\Gamma(K)=\Gamma_{\Rp{c}{g}{K}}$ admits a $T_2$-action that is free and transitive on vertices. Denote $\F_{g} := \F_{\Rp{c}{g}{K}}$ for $c$ regular. The homomorphism $\ti{K}\rightarrow K$ of (\ref{shorex}) restricts to a homomorphism $\ti{T}_2 \rightarrow T_2$ of $2$-torsion groups.

\begin{lem}\label{gtalk}
If $c \in T$ is regular, then the GKM-sheaf $\F_g := \F_{\Rp{c}{g}{K}}$ is $T_2$-equivariant, and the induction functor
$$Ind_{\ti{T_2}}^{T_2}: GKM_{\ti{T}_2}(\Gamma(\ti{K})) \mapsto  GKM_{T_2}(\Gamma(K))$$ 
satisfies $Ind_{\ti{T_2}}^{T_2}(\ti{\F_g}) \cong \F_g$, where $\ti{\F}_g = \F_{\Rp{c}{g}{\ti{K}}}$. 
\end{lem}

\begin{proof}
If $K \cong SU(2) \times U(1)^{r-1}$, then $T_2$ lies in the centre of $K$ so the action of $T_2$ on $\Rp{c}{g}{K}$ is simply by multiplying the zeroth entry, and this induces the action on $\F_g$.

For the remaining cases we use of Remark \ref{lracolercuh} to identify $$\F_g \cong (\ti{\F}_g \coprod \ti{\F}_g)/C_2 = Ind_{C_2}^{\Z_2}(\F_g)$$  
where the induction is with respect to the zero morphism $C_2 \rightarrow \Z_2$. Because $C_2 \subset \ti{T}_2$ is the kernel of the morphism $\ti{T}_2 \rightarrow T_2$, we can also say
$$\F_g \cong Ind_{\ti{T}_2}^{T_2}(\ti{\F}_g)$$
thereby making $\F_g$ equivariant. 
\end{proof}

\begin{prop}\label{waspsbumdie}
Denote $\F_{g} = \F_{\Rp{c}{g}{K}}$ where $c \in K$ is regular. Then $\F_g$ is isomorphic as a GKM-sheaf to the g-fold convolution product $\F_1 *...* \F_1$.
\end{prop}

\begin{proof}
For $K$ of type (i) we have (\ref{aolerucdlaorced}) so the result follows from  Lemmas \ref{insapcing} and \ref{overload} and Proposition \ref{dhiontoijw}. 

The result for general $K$ follows from the isomorphisms

$$ \F_1 * ... * \F_1 \cong Ind_{\ti{T_2}}^{T_2}(\ti{\F_1} *....* \ti{\F}_1 ) \cong Ind_{\ti{T_2}}^{T_2}(\ti{\F_g}) \cong \F_g$$
from Proposition \ref{supeagone} and Lemma \ref{gtalk}.	
\end{proof}

\begin{prop}\label{sdonomtehg}
If $K$ is of type (i) or (ii), then $\F_1$ is monodromy sheaf, with vertex stalk $\wedge(\lie{t}^*) \otimes A$ and holonomy $S_{\alpha(e)} \otimes Id_A$ for all non-degenerate edges $e$, where $S_{\alpha(e)} \in Aut(\wedge(\lie{t}^*))$ is induced by the Weyl reflection in the hyperplane $\alpha(e)^{\perp} \subset \lie{t}$.
\end{prop}

\begin{proof}
For type (i) the result follows from the $SU(2)$ case, Proposition \ref{jydiv}. In type (ii) we have $\F_1 = Ind_{\ti{T}_2}^{T_2}(\ti{\F}_1)$ where $\ti{\F}_1$ where is a monodromy sheaf as described above. Because the identification of the vertex stalks $\ti{\F}_1(v) \cong \wedge(\lie{t}^*) \otimes A$ were specifically chosen to be invariant under the $\ti{T}_2$-action, it is clear that $\F_1$ is also a monodromy sheaf with the required monodromy.
\end{proof}

\subsubsection{If c is not regular}

Let $c \in T$ be a not necessarily regular element with centralizer $Z(c) \subseteq K$ and define the Weyl group at $c$ by the formula $W_c := N_{Z(c)}(T)/T$. Because $K$ has semi-simple rank one, $W_c = W = \Z/2\Z$ if $c$ is not regular and $W_c$ is trivial if $c$ is regular. The $T$-action on $\Rp{c}{g}{K}$ extends to a $Z(c)$-action and there is a well known formula:
$$H_{Z(c)}^*(\Rp{c}{g}{K}) \cong H_T^*(\Rp{c}{g}{K})^{W_c}.$$ 

\begin{prop}\label{ghognofjiriston}	
Let $\mathcal{R} := \Rp{c}{g}{K}$ for $c$ not necessarily regular and let $\F_g$ be as defined in \S \ref{kgibiagbi}. Then $W_c$ acts on both $\F_{\mathcal{R}}$ and $\F_g$ giving rise to an isomorphism of $W_c$-invariants,

$$ H_T^*(\mathcal{R})^{W_c} \cong H^0(\F_{\mathcal{R}})^{W_c} \cong  H^0(\F_g)^{W_c}.$$
\end{prop}

\begin{proof}

If $W_c$ is trivial then $c$ is regular and the statement is vacuous. So assume $c$ is not regular so $W_c = W \cong\Z/2\Z$.

The twisted $W$-actions on $H_T^*(\mathcal{R})$ and on $\F_{\mathcal{R}}$ are the standard ones described in Proposition \ref{sonbocneyoga} determining the isomorphism $$H_T^*(\mathcal{R})^{W} \cong H^0(\F_{\mathcal{R}})^{W}$$

To construct the twisted $W$-action on $\F_g$, consider first the case of $K = \ti{K}\cong SU(2)\times U(1)^{r-1}$. In this case the representation variety splits as $\mathcal{R} \cong \Rp{c_1}{g}{SU(2)} \times \Rp{c_2}{g}{U(1)^{r-1}}$. If $\gamma: [0,1] \rightarrow T$ is a path connecting $c$ to a regular value, then Lemma \ref{corwhat} and Proposition \ref{supbut} imply that $\F_g$ acquires a twisted $W$-action and $H^0(\F_g)^{W} \cong H_T^*(\mathcal{R})^{W}$.

Now consider $K$ of type (ii) or (iii), fitting into the short exact sequence (\ref{shorex}). Denote $\mathcal{R} = \coprod_{\phi(\ti{c}) = c} \Rp{\ti{c}}{g}{\ti{K}}$  which by (\ref{galgoidsocov}) forms a $C_2^{g+1}$-Galois covering $ \ti{\mathcal{R}} \rightarrow \mathcal{R}$. The elements $\ti{c} \in \phi^{-1}(c)$ are either both regular or both non-regular. If they are both regular, then $\mathcal{R}$ is isomorphic as a $T$-space to $\Rp{c'}{g}{K}$ for regular $c' \in K$ and consequently $\F_{\mathcal{R}} \cong \F_g$ even before taking $W$-invariants. Otherwise, both $\ti{c} \in \phi^{-1}(c)$ are non-regular and $W$ acts on both cofactors, so $ H_T^*(\ti{\mathcal{R}})^W \cong H^0(\ti{\F}_g)^W \oplus H^0(\ti{\F}_g)^W$ since we have already confirmed this isomorphism for $K = \ti{K}$.

By Remark \ref{for later} the actions by $C_2^{g+1}$ and $W$ commute, on both $H_T^*(\mathcal{R})$ and on $\F_g$. Thus

\begin{align*} H_T^*(\mathcal{R})^W \cong (H_T^*(\ti{\mathcal{R}})^{C_2^{g+1}})^W   \cong (H_T^*(\ti{\mathcal{R}})^W)^{C_2^{g+1}} \cong (H^0(\ti{\F}_g)^W \oplus H^0(\ti{\F}_g)^W)^{C_2^{g+1}} \\ \cong ((H^0(\ti{\F}_g) \oplus H^0(\ti{\F}_g))^{C_2^{g+1}})^W \cong H^0(\F_g)^W.
	\end{align*} 
\end{proof}

\begin{rmk}\label{dgnongo4qn}
It is useful to describe more explicitly the twisted $W_c$-action on $\F_g$. This action is completely determined by its restriction to the vertex set $\V$. We may identify the vertex set $\V \cong \{ t \in T | t^2 \in c\}$ and $W_c$ acts on $\V$ by restricting the standard action on $T$. The stalks over every $ v \in \V$ are identified with $\F_g(v) = \wedge(\lie{t}^*)^{\otimes g} \otimes S(\lie{t}^*)$  (see Corollary \ref{toreghe}) on which $W_c$-acts by $S_{\alpha} \otimes S_{\alpha}$, where $S_{\alpha}$ denotes the extension of the Weyl reflection on $\lie{t}^*$ to algebra automorphisms of $\wedge(\lie{t}^*)$ and $S(\lie{t}^*)$. 
\end{rmk}

\subsection{Proving the theorems of \S \ref{mainresults}}\label{sect82}

In this section, let $K$ denote an arbitrary compact connected Lie group with maximal torus $T$. Denote by $\Delta$ the set of roots of $T\subset K$ and let $\Delta \subset \Pt$ be the set of roots modulo scalars.

For $\alpha \in \Pt$, denote by $K_{\alpha}$ the centralizer of $\ker(\alpha)$ in $K$. It is easy to see that

\begin{equation}\label{ZGR}
\Rp{c}{g}{K}^{\ker(\alpha)} = \Rp{c}{g}{K_{\alpha}} \subseteq \Rp{c}{g}{K}.
\end{equation}

\begin{lem}\label{ZGS}
The centralizer $K_{\alpha} \subseteq K$ is strictly larger than $T$ if and only if $\alpha \in \Delta$.
\end{lem}

\begin{proof}
Centralizers of tori in $K$ are connected (see \cite{MR781344} IV Theorem 2.3) so $K_{\alpha}$ is connected for any $\alpha \in \Pt$. Thus $K_{\alpha}$ is strictly larger than $T$ if and only the adjoint action of $\ker(\alpha) \subset T$ on $\lie{k}$ has fixed point set larger than $\lie{t}$. Since the roots $\Delta \subset \Lambda$ record the weights of the adjoint action, the result follows.
\end{proof}

Lemma \ref{ZGS} and (\ref{ZGR}) imply that the one-skeleton of $\Rp{c}{g}{K}$ is the pair

\begin{equation}\label{ZGT}
\Big( \bigcup_{\alpha \in \Delta} \Rp{c}{g}{K_{\alpha}},~\Rp{c}{g}{T}  \Big).
\end{equation}
Consequently, the construction of the GKM-sheaf $\F_{\Rp{c}{g}{K}}$ reduces to the torus and the semi-simple rank one cases.

\begin{proof}[Proof of Theorem \ref{thma}]
The description of the vertex set is from Corollary \ref{toreghe}. Lemma \ref{ZGS} reduces the description of the edge set reduces to the semisimple rank one case, which is described in Proposition \ref{bbboty}. $\Gamma_{\Rp{c}{g}{K}}$ is $T_2$-equivariant because $\Gamma_{\Rp{c}{g}{K_{\alpha}}}$ is for all $\alpha \in \Pt$.
\end{proof}

\begin{proof}[Proof of Theorem \ref{thmb}]
$\F_g := \F_{\Rp{c}{g}{K}}$. Restricting to vertices, there is a natural isomorphism 
\begin{equation}\label{dhf;oauhrgvbie}
\F_g|_{\V} \cong (\F_1)^{* g}|_{\V}
\end{equation}
as defined in Corollary \ref{toreghe}. For any $\alpha$, Proposition \ref{waspsbumdie} combined with (\ref{ZGR}) provides an isomorphism $ \F_g|_{\V \cup \E^{\alpha}} \cong (\F_1)^{* g}|_{\V \cup \E^{\alpha}}$ respecting (\ref{dhf;oauhrgvbie}). Gluing together completes the result.
\end{proof}

\begin{proof}[Proof of Theorem \ref{thmc}]
Let $\mathcal{R} = \Rp{c}{g}{K}$ where $c$ is not necessarily regular and let $\F_g$ be as defined in \S \ref{kgibiagbi}.	
	
The centralizer $Z(c)$ acts on $\mathcal{R}$ by conjugation, and this determines a twisted $W_c$-action on $\F_{\mathcal{R}}$ as described in Proposition \ref{sonbocneyoga}.

The action of $W_c$ on the restriction $\F_g|_{\V}$ is the one described in Remark \ref{dgnongo4qn}. Namely $W_c$ acts on $\V \cong \{ t \in T| t^2 = c\}$ by the restriction of the standard action on $T$, and $W_c$ acts on stalks $\F_g(v) \cong \wedge(\lie{t}^*)^{\otimes g} \otimes S(\lie{t}^*)$ by the tensor product of standard representations of $W_c$ on $\wedge(\lie{t}^*)$ and on $S(\lie{t}^*)$. To see that this action extends to all of $\F_g \cong (\F_1)^{* g}$, observe that $W_c$ respects the local system of the monodromy sheaf $\F_1$ because for any $w \in W_c$ and $e \in \E$ $$ (w \otimes w) (S_{\alpha(e)} \otimes Id_A) (w^{-1}\otimes w^{-1}) = S_{w\alpha(e) w^{-1}}\otimes Id_A = S_{\alpha(w \cdot e)} \otimes A$$ as an automorphism of $\wedge(\lie{t}^*)\otimes Id_A$.

Now let $i^*_{\alpha}: \F_{\mathcal{R}}(\V \cup \E^{\alpha}) \rightarrow \F_{\mathcal{R}}(\V)$ and $j^*_{\alpha}: \F_{g}(\V \cup \E^{\alpha}) \rightarrow \F_g(\V)$ denote the restriction maps and identify $$\F_g(\V) = \F_{\mathcal{R}}(\V) = \C\V \otimes \wedge(\lie{t}^*)^{\otimes g} \otimes S(\lie{t}^*).$$ By Proposition \ref{ghognofjiriston} we know that $\im(i^*_{\alpha})^{S_{\alpha}\otimes S_{\alpha}} = \im(j^*_{\alpha})^{S_{\alpha}\otimes S_{\alpha}} $, so 

\begin{align*} H^0(\F_{\mathcal{R}})^{W_c} \cong \Big(\bigcap_{\alpha \in \Delta} \im(i^*_{\alpha})\Big)^{W_c} = \Big(\bigcap_{\alpha \in \Delta} \im(i^*_{\alpha})^{S_{\alpha}} \Big)^{W_c} \\ = \Big(\bigcap_{\alpha \in \Delta} \im(j^*_{\alpha})^{S_{\alpha}} \Big)^{W_c} \cong H^0(\F_g)^{W_c}\end{align*} 
\end{proof}

Every connected compact Lie group $K$ possesses a finite covering group $\ti{K}\rightarrow K$ for which $\pi_1(\ti{K})$ is torsion-free. We can exploit this fact to simplify arguments, much as we did in the previous section.

\begin{lem}\label{suopy}
If $K$ is a compact, connected Lie group for which $\pi_1(K)$ is 2-torsion free, then for all roots $\alpha \in \Delta_+$, the fundamental group $\pi_1(K_{\alpha})$ is also 2-torsion-free. This implies that the exponential of the coroot $ exp ( \pi i h_{\alpha}) \in K$ does not equal the identity. 
\end{lem}

\begin{proof}
Consider the exact sequence of homotopy groups 
$$ \pi_2(K) \rightarrow \pi_2(K/K_{\alpha}) \rightarrow \pi_1(K_{\alpha}) \rightarrow \pi_1(K)$$
Here $\pi_2(K) = 0$ by Whitehead's theorem, and $\pi_2(K/K_{\alpha}) \cong H_2(K/K_{\alpha};\Z)$ is torsion free because $K/K_{\alpha}$ admits a Bruhat decomposition into even dimensional cells. Since $\pi_1(K)$ contains no 2-torsion, we deduce that $\pi_1(K_{\alpha})$ does not either.

Since $K_{\alpha}$ has semisimple rank one, it is isomorphic one of the groups in Lemma \ref{three} and is not equal $U(1)^{(r-1)}\times SO(3)$, so it is easy to verify that $exp(\pi i h_{\alpha}) \neq \id$.
\end{proof}

\begin{proof}[Proof of Theorem \ref{thmd}]
The restriction of $\F_{\Rp{c}{1}{K}}$ to $\V \cup \E^{\alpha}$ is naturally isomorphic to $\F_{\Rp{c}{1}{K_{\alpha}}}$ where $K_{\alpha}$ is of type (i) or (ii) so this follows directly from Proposition \ref{sdonomtehg}.
\end{proof}

\begin{proof}[Proof of Theorem \ref{thme}]
The restriction of $\F_{\Rp{c}{1}{K}}$ to $\V \cup \E^{\alpha}$ is naturally isomorphic to $\F_{\Rp{c}{1}{K_{\alpha}}}$  so this follows directly from Lemma \ref{gtalk}.
\end{proof}

\section{Cohomology calculations}\label{oughshogh}

For simplicity, we focus on simple groups $K$ for which $\pi_1(K)$ has finite and odd cardinality. Let $\Phi = Span_{\Z}\{ \pm h_{\alpha}|~ \alpha \in \Delta\}$ denote the coroot lattice of $K$ in $\lie{t}$. Lemma \ref{suopy} implies that the map $ \Phi \rightarrow T_2,~\xi \mapsto exp(\pi i \xi)$ defines an isomorphism

\begin{equation}\label{betrofgg}
 \Phi/ 2\Phi \cong T_2.
\end{equation}

\subsection{Regular c} 
Let $\F_1 = \F_{\rho}\cong \F_{\Rp{c}{1}{K}}$ is the monodromy GKM-sheaf of Theorem \ref{thmd}. Because $\F_{1}$ is $T_2$-equivariant, we may decompose $H^0(\F_1)$ into isotypical components  $$H^0(\F_1) \cong \bigoplus_{\chi \in \hat{T}_2} H^0(\F_1)^{\chi}. $$
The localization map $i^*:H^0(\F_1)=\F_1(\V\cup \E)  \rightarrow \F_1(\V)$ is injective and $T_2$-equivariant, so we may identify $H^0(\F_1)^{\chi}$ with the image of

$$i^*_{\chi}:H^0(\F_1)^{\chi} \hookrightarrow \F_1(\V)^{\chi} \cong \wedge(\lie{t}^*)\otimes A.$$
Employing Proposition \ref{noodelsouptime}, we have

\begin{equation}\label{gnboprnypanhg}
\im(i^*_{\chi}) = \bigcap_{\alpha \in \Delta} \im(i^*_{\alpha,\chi}),
\end{equation}
where $i^*_{\alpha,\chi}: \F_1(\V \cup \E^{\alpha})^{\chi} \hookrightarrow \F_1(\V)^{\chi}$ is the sheaf restriction.

Given $\alpha \in \Delta$ a root, choose a basis $\{ \alpha, \beta_1, ... ,\beta_{r-1}\}$ of $\clie{t}^*$ where the $\beta_i$ are orthogonal to $\alpha$ (equivalently, the $\beta_i$ are +1 eigenvectors for the reflection defined by $\alpha$). If $m \in \wedge(\lie{t}^*)$ is a monomial in this basis define $deg_{\alpha}(m) = 1$ if $\alpha$ is a factor of $m$ and $deg_{\alpha}(m)=0$ if not.

\begin{prop}\label{coodddot}
The image of $i^*_{\alpha,\chi}$ in $\wedge(\lie{t}^*)\otimes A$ is generated as an $A$-submodule by elements
$$ m \otimes 1 \text{  ~~~~~~~if $(-1)^{deg_{\alpha}(m)} \chi(exp(\pi i h_{\alpha})) = 1$  } $$
and
$$ m \otimes \alpha \text{    ~~~~~ if $(-1)^{deg_{\alpha}(m)} \chi(exp(\pi i h_{\alpha})) = -1$ } $$
as $m$ varies over a monomial basis of $\wedge (\lie{t}^*)$.

\end{prop}

\begin{proof}
The basis is chosen so that for a monomial $m \in \wedge(\lie{t}^*)$,  $$S_{\alpha}(m) = (-1)^{deg_{\alpha}(m)}m.$$

$T_2$ factors as $T_2' \times T_2''$, where $T_2''$ acts freely on $\E^{\alpha}$ and $T_2' \cong \Z/2\Z$ is generated by $exp(\pi i h_{\alpha})$, which fixes the edges in $\E^{\alpha}$ and transposes their incident vertices. Thus we may restrict attention to a single edge, and identify $i^*_{\alpha}$ with the intersection of the image of the matrix

$$	\left( \begin{array}{cc}
1 \otimes 1 & 1 \otimes \alpha  \\
S_{\alpha} \otimes 1 & - S_{\alpha} \otimes \alpha \end{array} \right)$$
in $(\wedge(\lie{t}^*) \otimes A)^{\oplus 2}$ with either the diagonal if  $\chi( exp(\pi i h_{\alpha})) =1 $ or the anti-diagonal if  $\chi( exp(\pi i h_{\alpha})) =-1$, from which the result follows.

\end{proof}

In \cite{bairdthesis} Chapter 8, it was shown that the projection map $\Rp{c}{g}{K} \rightarrow K^g$ sending $(k_0,....,k_g)$ to $(k_1,...,k_g)$ is a ``cohomological covering map''. In particular, this means that the induced map on cohomology is an injection $H_{T}^*(K^g) \hookrightarrow H_{T}^*(\Rp{c}{g}{K})$. The analogue of this for $\F_{\Rp{c}{g}{K}}$ is the following.

\begin{cor}\label{olarcelruhaoreh}
Let $\F_g$ denote the $g$-fold convolution product $\F_1 *...*\F_1$. Then the $T_2$ invariant part of cohomology $H^0(\F_g)^{T_2}$ is isomorphic to $H^*_T(K^g)$.	
\end{cor}

\begin{proof}
For $g=1$, just compare Proposition \ref{coodddot} with Example \ref{aoelublaoecbulcrbix}. For $g >1$ apply Proposition \ref{alcrldecrudi.}.
\end{proof}

\begin{rmk}\label{bnohgnpabnnth;nbnhfd}
Using Theorem \ref{gnboprnypanhg}, we identify $H_T^*(K) \cong H^0(\F_1)^{T_2}$ as a submodule of $\wedge(\lie{t}^*) \otimes S(\lie{t}^*)$. The equivariant fundamental class of $K$ spans the one-dimensional subspace $\wedge^r(\lie{t}^*) \otimes (\prod_{\alpha \in \Delta} \alpha) $. Compare this with Example \ref{aoelublaoecbulcrbix}.
\end{rmk}

Let $proj_r$ denote projection from $\wedge( \clie{t}^*) \otimes A $ onto $\wedge^r (\clie{t}^*) \otimes A$.

\begin{lem}\label{gobononhuebg}
Use Theorem \ref{gnboprnypanhg} to identify $H^0(\F_1)^{\chi}$ with a submodule of $\wedge(\lie{t}^*) \otimes A$. Then $$proj_r(H^0(\F_1)^{\chi}H^0(\F_1)^{\chi}) \subseteq \wedge^r (\clie{t}^*) \otimes (\prod_{\alpha \in \Delta}\alpha)A.$$ 
\end{lem}

\begin{proof}
The fact that $H^0(\F_1)$ is a $T_2$-equivariant algebra implies that $H^0(\F_1)^{\chi} H^0(\F_1)^{\chi} \subseteq H^0(\F_1)^{T_2} \subseteq \wedge(\lie{t}^*) \otimes A$. By Remark \ref{bnohgnpabnnth;nbnhfd}, we know $$proj_r( H^0(\F_1)^{T_2}) = H^0(\F_1)^{T_2} \cap (\wedge^r(\lie{t}^*) \otimes A) =  \wedge^r (\clie{t}^*) \otimes (\prod_{\alpha \in \Delta}\alpha)A.$$
\end{proof}

We will use this Lemma \ref{gobononhuebg} in combination with

\begin{lem}\label{mischief}
Let $d \in wedge^r (\clie{t}^*) \otimes A$ be a homogeneous element and let  $M$ be a free $A$-submodule of the $A$-algebra $\wedge (\clie{t}^*) \otimes A$ with homogeneous free generators $\{x_1,...,x_{2^r}\}$ satisfying $$ proj_r( x_i x_{2^r-j}) = a_{i,j}d$$ where $[a_{i,j}]$ is a non-singular matrix of complex numbers. Then $M$ is maximal among submodules satisfying $proj_r ( M M) \subseteq d\otimes A $ .  
\end{lem}

\begin{proof}
Suppose there exists $y \in (\wedge (\clie{t}^*) \otimes A) \setminus M$ for which $proj_r(My) \subseteq d \otimes A$. Then by subtracting an  $A$-linear combination of $x_i$ from $y$, we may assume that $proj_r(My)=0$. But this implies that $y = 0$ because $M$ is of maximal rank and the $proj_r$ is a nondegenerate pairing for $\wedge( \clie{t}^*) \otimes A $.  
\end{proof}

\subsection{Nonregular $c$}

For $c \in T \subseteq K$, the $W_c$-action on $\F_g$ defined in Remark \ref{dgnongo4qn} does not generally commute with the $T_2$-action. Instead $W_c$ acts by conjugation on $T_2$, restricting the standard action of $W_c$ on $T$. Consequently, $H^0(\F_g)^{W_c}$ decomposes into a sum over the orbit space $\hat{T}_2/W_c$,
\begin{equation}\label{imsbthigh}
H^0(\F_g)^{W_c} \cong \bigoplus_{[\chi] \in \hat{T}_2/W_c} (H^0(\F_g)^{W_c})^{[\chi]}.
\end{equation}

\begin{cor}\label{glncobnoonhcbnbvmnb,nffigh}
	Let $(W_{c})_{\chi}$ be the stabilizer of $\chi \in T_2$. If $H^0(\F_1)^{\chi}$ is free over $A$, then 
	$$ (H^0(\F_{g})^{W_c})^{[\chi]} \cong  (H^0(\F_1)^{\chi} \otimes_A...\otimes_A H^0(\F_1)^{\chi}) ^{(W_{c})_{\chi}}. $$
\end{cor}

\begin{proof}
The $W_c$-action permutes the summands of $\bigoplus_{\chi \in \hat{T}_2} H^0(\F_g)^{\chi}$ according to its action on $\hat{T}_2$, so $$(H^0(\F_{g})^{W_c})^{[\chi]} \cong  (H^0(\F_g)^{\chi}) ^{(W_{c})_{\chi}}$$
and the result then follows from the isomorphism $H^0(\F_g)^{\chi} \cong H^0(\F_1)^{\chi} \otimes_A ...\otimes_A H^0(\F_1)^{\chi} $ of Theorem \ref{bbnxnrpin} and Proposition \ref{alcrldecrudi.}.
\end{proof}

A choice of base vertex $v_*$ determines an isomorphism 
\begin{equation}\label{gnobnocmxnki}
\F_g(\V)^{\chi} \cong \F_g( v_*) = \wedge(\lie{t}^*)^{\otimes g} \otimes S(\lie{t}^*)
\end{equation} 
and we want to describe the twisted $(W_{c})_{\chi}$-action on $\F_g(\V)^{\chi}$ in terms of this identification.

\begin{lem}\label{gnnpvnlemenbnaonp}
Given a base vertex $v_*$, let $$\tau: {(W_{c})_{\chi}} \rightarrow Aut( \wedge(\lie{t}^*)^{\otimes g} \otimes S(\lie{t}^*))$$ denote the twisted action induced from (\ref{gnobnocmxnki}). Let $\tau^{st}$ denote the standard twisted $W$-action on $\wedge(\lie{t}^*)^{\otimes g} \otimes S(\lie{t}^*)$, induced by the standard Weyl group action on $\lie{t}$. Then $ \tau_w= \chi(t) \tau_w^{st}$ where $t \in T_2$ satisfies $t \cdot v_* = w \cdot v_*$. In particular, the restriction of $\tau$ to $(W_{c})_{\chi}$ is the standard action if either $(W_{c})_{\chi}$ fixes the base vertex $v_*$ or if $\chi$ is trivial.
\end{lem}

\begin{proof}
It was explained in the proof of Theorem \ref{bnnsopgnrpan} that $W$ acts on $\F_g(\V)^{\chi} \cong \C \V \otimes \wedge(\lie{t}^*)^{\otimes g} \otimes S(\lie{t}^*)$ by the tensor product of $\tau^{st}$ with the action on $\C \V$ induced by the action on $\V$. If $v_*$ is fixed by $(W_c)_{\chi}$ then the identification (\ref{gnobnocmxnki}) is undisturbed and $\tau = \tau^{st}$. Otherwise it is necessary to correct by multiplying by $t$, which introduces the factor $\chi(t)$.
\end{proof}

\subsection{Type $A_2$}

In this section we let $K= SU(3)$ or $PSU(3)$. Then because $\pi_1(K)$ has odd order, (\ref{betrofgg}) holds. Let $\{ e_1,e_2,e_3\}$ denote the standard basis in $\C^3$ with the standard pairing. Identify $\lie{t} = (e_1 +e_2+e_3)^{\perp} \subset \C^3$. We choose positive roots $\Delta_+ = \{\alpha_{i,j} = e_i -e_j| i<j\}$ and use the pairing to identify $\alpha_{i,j} = h_{\alpha_{i,j}}$. The group $\Phi/ 2 \Phi \cong T_2 \cong (\Z/2)^2$ has coset representatives consisting of $0$ and the three positive roots $\Delta_+$. There are three nontrivial characters $\chi_1,\chi_2,\chi_3$ of $T_2 \cong \Phi/2\Phi$, defined by $\chi_k([\alpha_{i,j}]) = 1$ if and only if $i,j,k$ are pairwise distinct. The Weyl group $S_3$ acts on the three nontrivial characters $\chi_1,\chi_2,\chi_3$ in the standard fashion.

\begin{prop}\label{sutwocase}
The summands $H^0(\F_1)^{\chi_k}$ are pairwise isomorphic free modules for $k=1,2,3$. 
Under the injection $$i^*: H^0(\F_1)^{\chi_3} \hookrightarrow \wedge \clie{t}^* \otimes S(\clie{t}^*)$$ described in Theorem \ref{gnboprnypanhg} and using coordinates $x_1 = \alpha_{1,3}$ and $x_2 = \alpha_{2,3}$, the free basis is

\begin{eqnarray}
	 1 \otimes x_1x_2,~~~~~ & & ~~~~~x_1\wedge x_2 \otimes (x_1-x_2),
	\\ x_1\otimes x_2 + x_2 \otimes x_1, & & x_1 \otimes x_2(x_1-x_2) - x_2 \otimes x_1(x_1-x_2).	
\end{eqnarray}

\end{prop}

\begin{proof}
That the $H^0(\F_1)^{\chi_k}$ are pairwise isomorphic follows from the fact that the $\chi_k$ are permuted by $W$.	

Let $M$ denote the submodule of $\wedge(\lie{t}^*) \otimes A$ generated by the purported free basis. It is straightforward linear algebra to check that $M$ is contained in the intersection $\bigcap_{\alpha \in \Delta} \im (i^*_{\alpha})$ of Theorem \ref{gnboprnypanhg}, so we conclude that $M \subseteq H^0(\F_1)^{\chi_3}$.  	

To prove $M \supseteq H^0(\F_1)^{\chi_3}$ observe that $M$ satisfies the conditions of Lemma \ref{mischief} with $d$  an element of $ \wedge^2(\lie{t}^*)\otimes x_1x_2(x_1 -x_2)$. Combining this with Lemma \ref{gobononhuebg} completes the proof.
\end{proof}

\begin{cor}\label{tobengonnslnt}
For $g\geq 1$, the $A$-module $H^0(\F_g)$ is free with Poincar\'e series $$ P_t(H^0(\F_g)) = ( (1+t^3+t^5+t^8)^g + 3(t^3+2t^4+t^5)^g)P_t(BT). $$
\end{cor}

\begin{proof}
By Corollary \ref{olarcelruhaoreh} we know $H^0(\F_g)^{T_2} = H_T^*(K^g)$ and this explains the term $(1+t^3+t^5+t^8)^gP_t(BT)$. The remaining three isotypical components $H^0(\F_g)^{\chi} \cong (H^0(\F_1)^{\chi})^{\otimes g}$ are isomorphic and have Poincar\'e series $(t^3+2t^4+t^5)^gP_t(BT)$ by Proposition \ref{sutwocase}.	
\end{proof}

\begin{cor}\label{gnbpyprong}
Let $\epsilon \in Z(K)$ and $\mathcal{R} = \Rp{\epsilon}{g}{K}$. Then $H^0(\F_R)^W$ is a free $A^W$-module with Poincar\'e series $$	( (1+t^3+t^5+t^8)^g + 
	(1+t^2+t^4)(t^3+2t^4+t^5)^g)P_t(BK) $$	
\end{cor}

\begin{proof}
Observe that every $\epsilon \in Z(K)$ has a square root in $Z(K)$, so the $W_{\epsilon} = W$ action fixes one vertex and permutes the remaining three. By Lemma \ref{gnnpvnlemenbnaonp} the induced $W$-action on $\wedge(\lie{t}^*)^{\otimes g} \otimes S(\lie{t}^*)$ is the standard one.	

Because $W$ permutes the three non-trivial characters $\{\chi_1,\chi_2,\chi_3\}$, $H^0(\F_{\mathcal{R}})^{W}$ decomposes into a sum of two $A^W$-modules as explained in Corollary \ref{glncobnoonhcbnbvmnb,nffigh}. The trivial character contributes the first summand $H^*_{T}(K^g)^W = H^*_{K}(K^g)$ with Poincar\'e series $$(1+t^3+t^5+t^8)^g P_t(BK).$$ 

The remaining summand is $$ (\bigoplus_{i=1}^3 H^0(\F_{\mathcal{R}}^{\chi_i}))^W \cong ((H^0(\F_1)^{\chi_3})^{\otimes g})^{W_{\chi_3}},$$ where $W_{\chi_3} = <(1,2)>$ is the stabilizer of $\chi_3$. The free basis of Proposition \ref{sutwocase} of $H^0(\F_1)^{\chi_3}$ is invariant under $(1,2)$ so the basis of $(H^0(\F_1)^{\chi_3})^{\otimes g}$ is also invariant and we deduce that 
\begin{align*}
P_t(((H^0(\F_1)^{\chi_3})^{\otimes g})^{W_{\chi_3}}) = (t^3 +2t^4 + t^5)^gP_t(A^{W_{\chi_3}}) =\frac{(t^3 +2t^4 + t^5)^g}{(1-t^2)(1-t^4)} \\ = (t^3 +2t^4+t^5)^g(1+t^2+t^4)P_t(BK)
\end{align*}
\end{proof}

\begin{proof}[Proof of Theorem \ref{firstone}]
In \cite{baird2009antiperfection} the Betti numbers of $H_{U(3)}^*( \Rp{\id}{g}{U(3)})$ were computed and it was shown to be torsion-free as an $H^*(BU(3))$-module. It is easily deduced that $H^*_{SU(3)}(\Rp{\id}{g}{SU(3)})$ is torsion-free and that $$P_t^{SU(3)}(\Rp{\id}{g}{SU(3)}) = ( (1+t^3+t^5+t^8)^g + 
(1+t^2+t^4)(t^3+2t^4+t^5)^g)P_t(BSU(3)).$$ By Theorem \ref{ifqfor} this means that there is an injective morphism of graded algebras $$H^*_{SU(3)}(\Rp{\id}{g}{SU(3)}) \hookrightarrow H^0( \F_{\Rp{\id}{g}{SU(3)}})$$
which by comparing Betti numbers must be an isomorphism, so $H^0( \F_{\Rp{\id}{g}{SU(3)}})$ is free by Corollary \ref{gnbpyprong}.
\end{proof}

\subsection{Type $B_2$}

The root system of $Spin(5)$ has positive roots $\{ 2e_1, 2e_2, -e_1+e_2, e_1+e_2\} \subset \C^2$, with corresponding coroots $\{ e_1, e_2, e_1+e_2, e_1-e_2\}$ under the standard pairing. Thus $T_2 \cong \Phi/2\Phi$ is generated by $[e_1]$ and $[e_2]$, so $[e_1 +e_2] = [e_1-e_2] \in \Phi/2\Phi$. The character group $\hat{T}_2$ has four elements defined in the basis $\{ [e_1], [e_2]\}$  by matrices $[1,1], [-1, -1], [1,-1], [-1,1]$, and the Weyl group interchanges only $[1,-1]$ with $[-1,1]$.

\begin{prop}\label{bbgqatioem}
For $K = Spin(5)$ and all $\chi \in \hat{T}_2$, the $A$-submodule $H^0(\F_1)^{\chi} \subset \wedge (\clie{t}^*) \otimes A$ is free. The generators of $H^0(\F_1^{[-1,-1]})$ are \begin{align*}
1\otimes e_1e_2, ~~~~~~~~~~~e_1\wedge e_2 \otimes (e_1+e_2)(e_1-e_2),~~~~\\ e_1\otimes e_2 + e_2\otimes e_1~~~~~~~ e_1\otimes e_2(e_1^2 -e_2^2) + e_2 \otimes e_1(e_2^2 - e_1^2).\end{align*}
Changing basis to $x= e_1+e_2$ and $y = e_1-e_2$, the generators of $H^0(\F_1^{[-1,1]})$ are

\begin{align*}
 1\otimes xy(x+y),~~~~~~~~~~~~~~~~~~~ x\wedge y \otimes (x-y), ~~~~\\
x\otimes y(x+y) + y \otimes x(x+y),~~~~~~~ x \otimes y(x-y) - y \otimes x(x-y).  
\end{align*}
\end{prop}  

\begin{proof}
Analogous to Proposition \ref{sutwocase}.
\end{proof}

\begin{cor}\label{sknalgbn}
For $K= Spin(5)$ and $g \geq 1$, the ring $H^0(\F_g)$ is a free module over $A$ with Poincar\'e series $$P_t(H^0(\F_g)) = (P_t(K)^g +  (t^3+t^4+t^6+t^7)^g+ 2 (t^4+ 2t^5+t^6)^g )P_t(BT). $$ Furthermore, under the augmentation morphism $ A \rightarrow \C$, the extension of scalars $H^0(\F_g) \otimes_{A} \C$ is a dimension $10g$ Poincar\'e duality ring over $\C$.
\end{cor}

\begin{proof}
The calculation of the Poincar\'e series is analogous to Corollary \ref{tobengonnslnt}. Since $H^0(\F_g)$ is free over $A$, we have
$$ P_t( H^0(\F_g) \otimes_{A} \C)  = P_t(H^0(\F_g))(1-t^2)^2 = (1+ t^3+t^7+t^{10})^g +  (t^3+t^4+t^6+t^7)^g+ 2 (t^4+ 2t^5+t^6)^g$$

The fundamental class for the Poincar\'e duality pairing a non-zero element of $$\Big((\wedge^{r}(\lie{t^*})\otimes \prod_{ \alpha_0 \in \Delta_+} \alpha_0)\Big)^{\otimes g} \in H^0(\F_1^{T_2})^{\otimes g} = H^0(\F_g^{T_2})$$ and the pairing is just the extension of scalars of $(proj_r)^{\otimes g}$. The generators listed in Proposition \ref{bbgqatioem} provide the nondegenerate pairs.	
\end{proof}

\begin{rmk}
For any regular $c \in Spin(5)$, $\Rp{c}{g}{Spin(5)}$ is equivariantly formal then its ordinary cohomology ring is isomorphic to $H^0(\F_g) \otimes_{A} \C$. Since we know that $\Rp{c}{g}{Spin(5)}$ is an orientable manifold of dimension $10g$, Corollary \ref{sknalgbn} is consistent with the conjecture that $\Rp{c}{g}{Spin(5)}$ is equivariantly formal.
\end{rmk}

In terms of the basis $e_1,e_2$, the Weyl group for type $B_2$ is generated by (any two of) the reflections:

$$s_1:= \begin{pmatrix} -1 & 0 \\ 0 & 1 \end{pmatrix}, s_2:=\begin{pmatrix} 1 & 0 \\ 0 & -1 \end{pmatrix}, s_3:= \begin{pmatrix} 0 & 1 \\ 1 & 0 \end{pmatrix}$$

\begin{cor}
Let $\epsilon \in Spin(5)=K$ be the nontrivial central element and let $\mathcal{R} = \Rp{\epsilon^i}{g}{K}$ for $i=0$ or $1$. The ring of Weyl invariants $H^0(\F_{\mathcal{R}})^{W_{\epsilon^i}}$ is a free module over $A^W = H^*(BK)$ with Poincar\'e series 
$$( P_t(K)^g +(t^3+t^4+t^6+t^7)^g + t^2(1+t^4)(t^4+2t^5+t^6)^g)P_t( BK)$$ if $g+i$ is even and $$ ( P_t(K)^g + t^4(t^3+t^4+t^6+t^7)^g + t^2(1+t^4)(t^4+2t^5+t^6)^g)P_t( BK)$$ if $g+i$ is odd.
\end{cor}

\begin{proof}
The generators of $H^0(\F_1)^{[-1,-1]}$ are not invariant under the standard $W_{\epsilon^i} =W$ action, but instead are weight vectors, of weight $-1$ for $s_1$ and $s_2$ and weight $+1$ for $s_3$. Taking tensor powers, it follows that for even $g$, the free $A$-generators of $H^0(\F_g)^{[-1,-1]}$ equal the free $A^W = H^*(BK)$ generators of $(H^0(\F_g)^{[-1,-1]})^{W_{\id}}$, while for odd $g$ the generators of $(H^0(\F_g)^{[-1,-1]})^{W_{\id}}$ are $e_1 e_2$ times the generators of $H^0(\F_g)^{[-1,-1]}$. On the other hand, one checks that under the twisted action of $W_{\epsilon}$ on $H^0(\F_1)^{[-1,-1]}$ the generators are invariant, so for odd $g$ the generators of $(H^0(\F_g)^{[-1,-1]})^{W_{\epsilon}}$ coincide with those of $H^0(\F_g)^{[-1,-1]}$ while for even $g$ we must shift by $e_1 e_2$.

The remaining pair of characters are permuted by the Weyl group, so we must consider $(H^0(\F_g^{[1,-1]}) \oplus H^0(\F_g^{[-1,1]}) )^{W_{\epsilon}} \cong H^0(\F_g^{[1,-1]})^{W_{[1,-1]}}$, which will be a free module over $A^{W_{[1,-1]}}$. Similar considerations now apply, but in this case generators must be shifted by a degree two element (either $e_1$ or $e_2$ depending on the parity of $g+i$). 
\end{proof}

\subsection{Type $G_2$}

The compact group $G_2$ has root system spanning the orthogonal complement of $e_1+e_2+e_3$ in $\C^3$, with simple roots $\alpha = e_1-e_2$ and $\beta = -2e_1+e_2+e_3$ and remaining positive roots $\alpha+\beta$, $2\alpha +\beta$, $3\alpha+\beta$ and $3\alpha + 2 \beta$. The Weyl group is isomorphic to the dihedral group $D_6$ and acts transitively on the nontrivial characters of $T_2 \cong \Phi/2\Phi$, with stabilizer isomorphic to $\Z_2\oplus \Z_2$. If we let $\chi$ denote the nontrivial character sending $[h_{\alpha}]$ and $[h_{3\alpha + 2\beta}]$ to $1$ then $W_{\chi}$ is generated by matrices    

$$s_1:= \begin{pmatrix} 0 & 1 \\ 1 & 0 \end{pmatrix}, s_2:=\begin{pmatrix} -1 & 0 \\ 0 & -1 \end{pmatrix}$$
in the basis $x = \alpha +\beta$ and $y=2\alpha +\beta$.

\begin{prop}\label{gnnboynpqjgnspdfjgnnhj}
For $K=G_2$, the module $H^0(\F_1)^{\chi}$ is a free $A$-module generated by

\begin{align*}
 1 \otimes xy(2x-y)(2y-x),~~~~~~~~~~~~~~~~~~~~~~(x \otimes y - y \otimes x)(x^2-y^2),~~~~~~~ \\ x\wedge y \otimes (x^2-y^2),~~~~~~~~~~~~~~~~~~~~~~~~~~(x \otimes y + y \otimes x)(2x-y)(2y-x).
\end{align*}
\end{prop}

\begin{proof}
Analogous to Proposition \ref{sutwocase}.
\end{proof}

\begin{cor}
For $K = G_2$, the graded ring $H^0(\F_g)$ is a free $A$-module  with Poincar\'e series $ (P_t(K)^g + 3(t^6+2t^7+t^8)^g)P_t(BT)$. Moreover the extension of scalars by the augmentation map $ H^0(\F_g) \otimes_{A} \C$ is a Poincar\'e duality ring of dimension 14$g$.
\end{cor}

\begin{proof}
Analogous to Corollary \ref{sknalgbn}.
\end{proof}

The simply connected group $G_2$ has trivial centre. 

\begin{cor}
For $K = G_2$, the ring $H^0(\F_{\Rp{\id}{g}{K}})^W \cong (H^0(\F_g))^W$ is a free $A^W$-module with Poincar\'e series $$ (P_t(K)^g + \frac{1}{2}(1+t^4+t^8)(t^6+t^7)^g ( (t^4+t^2)(t+1)^g + (t^2-1)(t-1)^g)) P_t(BK).$$
\end{cor}
\begin{proof}
We have $P_t(H^0(\F_g)^W) = P_t(K^g)P_t(BK) + P_t( (H^0(\F_1)^{\chi})^{\otimes g})^{W_{\chi}}) $. The formula follows by computation from Proposition \ref{gnnboynpqjgnspdfjgnnhj}.
\end{proof}

\subsection{Type $A_3$}\label{typea3}

Let $\{ e_1,e_2,e_3,e_4\}$ denote the standard basis in $\C^4$ with the standard pairing. Identify $\lie{t} = (e_1 +e_2+e_3 + e_4)^{\perp} \subset \C^3$. We choose positive roots $\Delta_+ = \{\alpha_{i,j} = e_i -e_j| i<j\}$ and use the pairing to identify $\alpha_{i,j} = h_{\alpha_{i,j}}$. The Weyl group $W = S_4$ acts on the group of characters of $T_2 = \Phi/2\Phi$ with two nontrivial orbits, distinguished by where they send $-\id$.

\begin{prop}
Let $\chi : \Phi/2\Phi \rightarrow \Z/2$ be the nontrivial character sending $[h_{\alpha_{1,2}}], [h_{\alpha_{1,3}}], [h_{\alpha_{2,3}}]$ to $1$ ( $\chi$ lies in orbit of size 4). Then $H^0(\F_1)^{\chi}$ is a free $A$-module. Choosing coordinates $x_i = \alpha_{i,4}$, $i=1,2,3$, we have free basis 

\begin{eqnarray*}
	1 \otimes x_1 x_2 x_3,~~~ & (x_1\wedge x_2 \wedge x_3) \otimes (x_1-x_2)(x_2-x_3)(x_3-x_1),\\ 
	\sum_{c.p} x_1 \otimes x_2x_3,~~ & \sum_{c.p.} (x_1 \wedge x_2) \otimes (x_1-x_2)x_1x_2x_3,\\
	\sum_{c.p.} x_1 \otimes x_1x_2x_3, ~&\sum_{c.p.} (x_1 \wedge x_2) \otimes (x_1-x_2)x_2x_3,\\
	\sum_{c.p.} x_1\otimes (x_2x_3)^2, &\sum_{c.p.} (x_1 \wedge x_2) \otimes (x_1-x_2)x_3.\\	
\end{eqnarray*}

where the sums are over cyclic permutations by the 3-cycle $(1,2,3)$.
\end{prop}

\begin{proof}
Analogous to Proposition \ref{sutwocase}.
\end{proof}

On the other hand, the remaining Weyl orbit of characters gives rise to non-free modules:

\begin{prop}
Let $\chi : \Phi/2\Phi \rightarrow \Z/2$ be the non-trivial character satisfying $\chi(h_{\alpha_{1,2}}) = \chi(h_{\alpha_{3,4}}) = 1$. Then $H^0(\F_1)^{\chi}$ is not free over $A$. The Hilbert series of $H^0(\F_1)^{\chi}$ is $$P_t(H^0(\F_1)^{\chi}) \cong \frac{(2t^7+4t^8+2t^9+t^{10}+t^{11}-t^{12}-t^{13})}{(1-t^2)^3}.$$
\end{prop}
\begin{proof}
The calculation was done using MAGMA \cite{bosma1997magma} (explained in Appendix \ref{magma}). Because the numerator of the Hilbert series is has negative coefficients, the module cannot be free.
\end{proof}

This example disproves the general conjecture that $\Rp{c}{g}{K}$ is equivariantly formal for all $c$ and $K$, because it fails when $K = SU(4)$ and $c$ is regular.

\subsection{Data tables}\label{aorecdulrocedj}

The following tables list the Hilbert series of $H^0(\F_1)$ and whether or not it is a free $A$-module, for several simply connected compact Lie groups $K$. All results not already described in \S \ref{oughshogh} were obtained using MAGMA \cite{bosma1997magma}.

The terms correspond to into Weyl orbits of $T_2$-isotypical components of $H^0(\F_1)$. Every term containing no negative coefficients corresponds to a free $A$-module summand.
\\

\begin{tabular}{ |c |c| p{13cm}|}
\hline
Lie type  & Free? &  Hilbert series of $H^0( \F_{\Rp{c}{1}{K}})$ times $(1-t)^{\rk( K)}$ for regular $c$.

\\

\hline 

$A_2$ & yes & $ P_t(A_2) + 3(t^3+2t^4+t^5)$

\\

$B_2$ & yes & $P_t(B_2) + 2 (t^4+ 2t^5+t^6) + (t^3+t^4+t^6+t^7)$

\\

$G_2$  &yes& $P_t(G_2) + 3(t^6+2t^7+t^8)$

\\

$A_3$  &no& $P_t(A_3) + 4(1+t)^2(t^8+t^5) +3 (1+t)(-t^{12}+t^{10}+2t^8+ 2t^7)$

\\

$B_3$  &no& $P_t(B_3) + 3t^7(t+1)(1+2t^3+t^6)+4t^9(t+1)^2(t^3+1)(-t^3+t^2+1)$

\\

$C_3$  &no&$ P_t(C_3) + 3t^8(t+1)^2(t^3+1)+3t^9(t+1)(-t^9 + 2t^5 + t^2 + t + 1)+ t^5(t+1)(t^3+1)(t^7+1)$

\\

$A_4$  &no& $P_t(A_4) + 5t^7(t+1)^2(t^3+1)(t^5+1)+ 10t^{11}(t+1)^2(t^3+1)(-t^4+t+2)$

\\

$B_4$  &no& $P_t(B_4) + 4t^{11}(t+1)(t^3+1)^2(t^7+1)+3t^{15}(t+1)(t^3+1)(-t^{11} + t^7 + 2t^3 + t^2 + 1)+8t^{16}(t+1)^2(t^3+1)(-t^8+t^5+t^2+1)$

\\

$C_4$  &no& $P_t(C_4)+ 4t^{12}(t+1)^2(t^3+1)(t^7+1)+t^7(t+1)(t^3+1)(t^7+1)(t^{11}+1)+6t^{16}(t+1)^2(-t^{13} + 3t^5 + t^2 + 1)+4t^{15}(t+1)^2(t^3+1)^2(-t^9 +t^8 - t^7 + 2t^6 - t^5 + t^4 - 2t^3 + t^2 + 1)$

\\
		
$D_4$ &no& $P_t(D_4) + 12t^{11}(t+1)^2(t^3+1)(-t^8+t^7+t^2+1) + 3t^{12}(t+1)^2(-3t^9 + 6t^5 + t^2 - t + 1)$

\\

$F_4$  &no& $ P_t(F_4) + 3t^{15}(t+1)(t^3+1)(t^7+1)(t^{11}+1) + 11t^{24}(t+1)^2(t^3+1)(-t^{14} + t^{11} - t^8 + t^6 + t^5 - t^3 +t^2 + 1)$

\\
\hline
\end{tabular}
\\
\\
\\
In the following table we collect Hilbert series for $H^0(\F_{\Rp{\id}{1}{K}})^W$ divided by $ P_t(BK)$, where $\id \in K$ is the identity. 
\\

\begin{tabular}{ |c | c| p{13cm}|}
\hline

Lie type & Free? & Hilbert series for $H^0(\F_{\Rp{\id}{1}{K}})^{W}$ divided by $P_t(BK)$ 

\\
\hline

$A_2$ & yes & $P_t(A_2) + (t^4+t^2+1)(t^5+2t^4+t^3)$

\\

$B_2$ & yes & $P_t(B_2) + t^4(t^7+t^6+t^4+t^3) + t^2(t^4+1)(t^6+2t^5+t^4)$

\\

$G_2$ & yes & $P_t(G_2) + \frac{1}{2}t^6(t+1)(t^8 +t^4+1)(t^5+t^4+2t^3-t+1)$

\\
\hline
\end{tabular}

\begin{appendix}\label{magma}
\section{MAGMA Calculation}

To compute the Hilbert series of $H^0(\F_1)$ and to check whether it is free over $A$, we use a Magma program. The program uses Magma functions for Coexeter groups and root systems. The program begins by defining an $A$-module $M$ isomorphic to $\wedge(\lie{t}^*) \otimes A$. Then it constructs the submodules $\im(i^*_{\chi,\alpha})$ using the generators described in Proposition \ref{coodddot}.  Finally, it intersects these submodules to get $\im(i^*_{\chi}) = H^0(\F_1)^{\chi}$ and applies the functions IsFree and HilbertSeries, to test freeness and obtain the Hilbert series. 

\end{appendix}

\bibliographystyle{amsalpha}
\bibliography{TomReferences}

\end{document}